\documentclass{daj}
\usepackage[T1]{fontenc}
\usepackage{microtype}
\usepackage{amsmath}
\usepackage{amssymb}
\usepackage{amsthm}
\usepackage{amscd}
\usepackage{enumitem}
\usepackage{mathscinet}
\usepackage[giveninits=true,doi=false,url=false,sortcites,backend=biber,backref=true]{biblatex}
\makeatletter
\newtheorem*{rep@theorem}{\rep@title}
\newcommand{\newreptheorem}[2]{%
\newenvironment{rep#1}[1]{%
 \def\rep@title{#2 \ref{##1}}%
 \begin{rep@theorem}}%
 {\end{rep@theorem}}}
\makeatother

\newtheorem{theorem}{Theorem}[section]
\newreptheorem{theorem}{Theorem}
\newtheorem{lemma}[theorem]{Lemma}

\newtheorem{cor}[theorem]{Corollary}

\theoremstyle{definition}
\newtheorem{definition}[theorem]{Definition}

\newtheorem{claim}{Claim}

 \newenvironment{claimproof}{\begin{proof}}{\end{proof}}

\makeatletter

\def\dotminussym#1#2{%
  \setbox0=\hbox{$\m@th#1-$}%
  \kern.5\wd0%
  \hbox to 0pt{\hss\hbox{$\m@th#1-$}\hss}%
  \raise.6\ht0\hbox to 0pt{\hss$\m@th#1.$\hss}%
  \kern.5\wd0}

\mathchardef\mhyphen="2D

\dajAUTHORdetails{%
  title = {An Analytic Approach to Sparse Hypergraphs: Hypergraph Removal}, 
  author = {Henry Towsner},
  plaintextauthor = {Henry Towsner},
  plaintexttitle = {An Analytic Approach to Sparse Hypergraphs: Hypergraph Removal}, 
  runningtitle = {An Analytic Approach to Sparse Hypergraphs}, 
  runningauthor = {Henry Towsner},
  copyrightauthor = {H. Towsner},
  keywords = {graph limit, pseudorandom graphs, hypergraph removal},
}   

\dajEDITORdetails{%
   year={2018},
   number={3},
   received={12 January 2017},   
   revised={10 October 2017},    
   published={24 January 2018},  
   doi={10.19086/da.3104},       
}   

\begin{document}

\begin{frontmatter}[classification=text]

\title{An Analytic Approach to Sparse Hypergraphs: Hypergraph Removal}
\author[hpt]{Henry Towsner}

\begin{abstract}
The use of tools from analysis to address problems in graph theory has become an active area of research, using the construction of analytic limit objects from sequences of graphs and hypergraphs.  We extend these methods to sparse but pseudorandom hypergraphs.  We use this framework to give a proof of hypergraph removal for sub-hypergraphs of sparse random hypergraphs.
\end{abstract}

\end{frontmatter}

\section{Introduction}

In this paper we bring together two recent trends in extremal graph theory: the study of ``sparse random'' analogs of theorems about dense graphs, and the use of methods from analysis and logic to handle complex dependencies of parameters.  To illustrate these methods, we will prove a version of the Hypergraph Removal Lemma for dense sub-hypergraphs of sparse but sufficiently pseudorandom hypergraphs.

\subsection{What is Sparse Hypergraph Removal?}

The original removal theorem was Rusza and Szemer\'edi's Triangle Removal Lemma \cite{rusza:MR519318}, which states:
\begin{theorem}
  For every $\epsilon>0$, there is a $\delta>0$ such that whenever $G\subseteq {V\choose 2}$ is a graph with at most $\delta|V|^3$ triangles, there is a set $C\subseteq G$ with $|C|\leq\epsilon|V|^2$ such that $G\setminus C$ contains no triangles at all.
\end{theorem}
This result was later extended to graphs other than triangles \cite{erdos:MR932119}, and ultimately to hypergraphs \cite{frankl:MR1884430,MR2373376,nagle:MR2198495}.  All these arguments depend heavily on the celebrated Szemer\'edi Regularity Lemma \cite{MR540024}, and its generalization, the hypergraph regularity lemma \cite{rodl:MR2069663,MR2373376}.  (Recently, Fox \cite{fox:MR2811609} has given a proof of graph removal without the use of the regularity lemma, which gives better bounds as a result.)

This statement is only meaningful for dense graphs, when $|G|>\epsilon|V|^2$, since otherwise we could simply remove all of $G$.  Various generalizations extend this to sparser graphs by considering ``relative graph removal'': we consider the case where $G$ is contained in an ambient graph $\Gamma$, where $\Gamma$ is sparse, and we measure sizes relative to $\Gamma$.  To make sense of this for arbitrary $\epsilon$, we need to consider a sequence of ambient graphs $\{\Gamma_n\}$.  Relative triangle removal for $\{\Gamma_n\}$ states:
\begin{quote}
    For every $\epsilon>0$, there are $\delta>0$ and $N$ such that whenever $n\geq N$ and $G\subseteq \Gamma_n$ is such that $|\{\text{triangles in }G\}|<\delta|\{\text{triangles in }\Gamma_n\}|$, there is a set $C\subseteq G$ with $|C|\leq\epsilon|\Gamma_n|$ such that $G\setminus C$ contains no triangles at all.
\end{quote}
Conventional triangle removal is the case where each $\Gamma_n$ is the complete graph on $n$ vertices.  The main case that has been studied is when $\Gamma_n$ is chosen to be a random graph on $n$ vertices which is ``not too sparse''---in practice, chosen so that it has roughly $n^{2-1/r}$ edges for some $r$.  This generalization was proven, essentially, by Kohayakawa, \L{}uczak, and R\"odl \cite{kohayakawa:MR1661982} and stated directly in this form in \cite{MR2815608}.  Their proof passes through a \emph{pseudorandomnes} assumption: they prove triangle removal relative to $\Gamma_n$ when $\Gamma_n$ is \emph{$(p,\gamma p^3n)$-bijumbled} where $p\geq n^{-1/2}$, and it is known that when $\Gamma_n$ is chosen randomly with $n^{2}p$ edges then, with high probability, $\Gamma_n$ is $(p,\gamma p^3n)$-bijumbled.

Various generalizations to other graphs and ultimately to arbitrary hypergraphs have appeared in the literature \cite{kohayakawa:MR1379396,green:MR2415379,ziegler08,MR2279549, ConlonFoxZhao,MR3385638,MR3273450,MR3327533}.  The most general results of this kind have roughly the form:
\begin{theorem}\label{thm:random_removal}
  For every $k$-uniform hypergraph $K$ with $v$ vertices, there is an $r$ so that, for every every $\epsilon>0$ there is an $N$ and a $\delta>0$ so that whenever $\Gamma$ is a random graph on $n\geq N$ vertices chosen so that each $k$-tuple is an edge in $\Gamma$ with probability $n^{-1/r}$, chosen independently, then with probablity $\geq 1-\epsilon$, whenever $A\subseteq\Gamma$ with $\frac{hom(K,A)}{|{n\choose v}|}<\delta$, there is an $L\subseteq A$ with $|L|\leq\epsilon|\Gamma|$ so that $hom(K,A\setminus L)=0$.
\end{theorem}
There are also a number of closely related results on sparse sets of integers \cite{MR2415379,MR2279549,ziegler08}.

Our main result is another proof of this theorem.  We have not attempted to extract explicit bounds (either on the dependence of $r$ on $K$ or the dependence of $N,\delta$ on $\epsilon$); since limiting arguments of the kind we used here do not lend themselves to extracting optimal bounds, we expect they would be somewhat worse than those given by the combinatorial proofs cited above.

As we will discuss in more detail below, it is by now well-known that hypergraph removal and regularity lemmas are closely related to the structure of $\sigma$-algebras in the structures obtained by taking limits of finite hypergraphs.  Our main goal in this paper is to extend this relationship to the sparse setting.  Towards this end, our proof uses a new notion of pseudorandomness (``correctly counting copies'') which is essentially an approximate version of Fubini's theorem, and precisely guarantees that our limiting structures satisfy Fubini's theorem for certain integrals.

\subsection{The Analytic Approach}

Our arguments in this paper draw on recent developments in analytic approaches to graph theory.  Probably the most widely studied approach is the method of \emph{graph limits} and \emph{graphons} introduced by Lov\'asz and coauthors \cite{borgs:MR2455626,lovasz:MR2306658,lovasz:MR2274085}.  Related approaches go back to work on exchangeable sequences of random variables beginning with Aldous and Hoover \cite{diaconis:MR2463439,aldous:MR637937,hoover:arrays,kallenberg:MR2161313}, and more recently, similar methods have been studied by Hrushovski, Tao, and others \cite{austin:MR2666763,hrushovski,goldbring:_approx_logic_measure}.  Analytic proofs of regularity and removal lemmas have been given using all these methods \cite{elek07,MR2964622,tao07,tao:MR2259060,towsner:modeltheoretic}.  These techniques obtain a correspondence between a sequence of arbitrarily large finite graphs on the one hand, and some sort of infinitary structure on the other.  Statements about density fit naturally in these frameworks since the normalized counting measure on large finite graphs corresponds to an ordinary measure on the infinitary structure.

In this paper, we describe a similar correspondence which applies to sub-hypergraphs of sparse, pseudorandom hypergraphs.  In the finite setting, the natural replacement for the normalized counting measure is the counting measure normalized by the ambient hypergraph.  This introduces new complications in the infinitary world: we end up with a natural measure on sets of $k$-tuples which is not a genuine product measure.  (This perspective on the problem was suggested to us by Hrushovski.)  In place of a single measure, we end up with a family of measures, and the pseudorandomness from the finitary setting is used to ensure that this family of measures obeys certain compatibility properties.

We use this method to give an analytic proof of sparse hypergraph removal.  Our approach to hypergraph removal depends heavily on the use of the Gowers uniformity (semi)norms \cite{MR1844079}.  As Conlon and Gowers point out \cite{conlon:2011}, such an approach cannot hope to give optimal bounds, and, relatedly, depends on a stronger notion of pseudorandomness than strictly needed.  We stick to this method both because we believe these norms are interesting in their own right, and because we believe it illustrates the analytic approach to sparse hypergraphs more clearly than an attempt to derive optimal bounds would.

Because of the analytic nature of our proof, the pseudorandomness property we need says that $\Gamma$ should have ``measure-theoretic'' properties which resemble those of the complete hypergraph.  Whenever $(V,E)$ is a small hypergraph, we write $\Gamma^V_E=hom((V,E),\Gamma)$ for the collection of all homomorphic copies of $(V,E)$ in $\Gamma$.  (This and the other notation in this paragraph will be introduced more rigorously in Sections \ref{notation} and \ref{families}.)  There is a natural way to choose a random element of $\Gamma^V_E$---uniformly---corresponding to a natural measure on $\Gamma^V_E$---the normalized counting measure.  However, we often obtain copies of $V$ by partitioning $V=V_0\cup V_1$, first choosing a copy $x_{V_0}\in\Gamma^{V_0}_{E}$, and then later choosing an extension $x_V$ such that $x_V\upharpoonright V_0=x_{V_0}$.  It will be convenient to have a notation for this: $\Gamma^{V_1}_{E,x_{V_0}}$ is the set of $x_{V_1}$ such that $x_{V_1}\cup x_{V_0}\in\Gamma^V_{E}$.  This leads to a second approach to choosing random elements of $\Gamma^V_E$: first choose $x_{V_0}\in \Gamma^{V_0}_{E}$ uniformly, then choose $x_{V_1}\in\Gamma^{V_1}_{E,x_{V_0}}$ uniformly.  There is a measure on $\Gamma^V_E$ corresponding to this approach, and in general, when $\Gamma$ is sparse, these measures might disagree.  (We give an example in Section \ref{pseudorandomness}.)

When $\Gamma$ is sufficiently pseudorandom, however, these measures agree, and this is essentially the property we need.  More precisely, we need to consider choosing hypergraphs by a three step process: that whenever $(V\cup W,E)$ is a small hypergraph, for most $a_W$ (according to the normalized counting measure on $\Gamma^W_{E}$), all possible partitions $V=V_0\cup V_1$ induce the same measure on $\Gamma^{V}_{E,a_W}$.  Roughly speaking, we want to say that for almost every $a_W\in\Gamma^W_E$, for almost every $x_{V_0}\in\Gamma^V_{E,a_W}$,
\[|\Gamma^{V_1}_{E,x_{V_0}\cup a_W}|\approx\frac{|\Gamma^V_{E,a_W}|}{|\Gamma^{V_0}_{E,a_W}|}.\]
We define this notion, that \emph{$(G,\Gamma)$ $\delta$-consistently counts copies of $(U,E)$}, precisely in Section \ref{pseudorandomness}.

With this we can state our main theorem:
\begin{theorem}\label{thm:main}
  For every $k$-uniform hypergraph $K$ on vertices $V$ and every constant $\epsilon>0$, there are $\delta,\zeta$ so that whenever $\Gamma$ is a $\zeta,|K|2^{2k}$-ccc $k$-uniform hypergraph and $A\subseteq \Gamma$ with $\frac{hom(K,A)}{|\Gamma^V_K|}<\delta$, there is a subset $L$ of $A$ with $|L|\leq\epsilon|\Gamma|$ such that $hom(K,A\setminus L)=0$.
\end{theorem}
Together with Theorem \ref{thm:ccc}, which states that when $\Gamma$ will be $\zeta,|K|2^{2k}$-ccc with high probability when chosen randomly and sufficiently densely, these give a proof of Theorem \ref{thm:random_removal}.

In \cite{goldbring:_approx_logic_measure}, Isaac Goldbring and the author proposed a general framework for handling analytic arguments of the sort in this paper, which we called \emph{approximate measure logic}.  In this paper, there is no assumption that the reader is familiar with that particular framework, but we pass quickly over the logical preliminaries, and refer the reader to that paper for more detailed exposition.

\subsection{Outline and Comparison to \textit{On the Triangle Removal Lemma for Subgraphs of Sparse Pseudorandom Graphs} \cite{MR2815608}}

Section \ref{families} introduces our general notation for families of measures concentrating on configuations that depend on the ambient graph $\Gamma$, since virtually all subsequent work requires being able to discuss these notions.

Section \ref{sigma_algebras} introduces the $\sigma$-algebras $\mathcal{B}_{V,\mathcal{I}}$ which play a central role in our proof.  In the standard proof of triangle removal in the dense setting, one uses graph regularity to produce a partition into components which have a certain pseudorandomness property, sometimes called DISC.  In our setting, the analog of hypergraph regularity is the observation that we may decompose a function as
\[f=\mathbb{E}(f\mid\mathcal{B}_{V,\mathcal{I}})+(f-\mathbb{E}(f\mid\mathcal{B}_{V,\mathcal{I}})),\]
and our analog of DISC is the randomness property satisfied by the second term $g=f-\mathbb{E}(f\mid\mathcal{B}_{V,\mathcal{I}})$, that $||\mathbb{E}(g\mid\mathcal{B}_{V,\mathcal{I}})||_{L^2(\mu^V_E)}=0$.  (See also \cite{MR3583029} for more about the connection between these $\sigma$-algebras and hypergraph regularity.)

In Section \ref{hypergraph_removal}, we carry out the first step of our proof.  We will say that a probability measure \emph{has regularity} if, roughly speaking, the randomness property DISC implies a somewhat stronger randomness notion---we refer to this as ``having regularity'' because it implies that the partition given by the regularity lemma actually has the properties we expect it to have.  The main result of that section, Theorem \ref{counting}, is closely analogous the counting lemma that appears in most proofs of hypergraph removal, and also to \cite[Lemma 11]{MR2815608}.  It completes the proof of hypergraph removal subject to the requirement that our measures have regularity.  We conclude the section by noting that the measures needed in the dense setting do have regularity, giving a proof of hypergraph removal in the dense setting.

We then turn to establishing the infinitary setting needed to complete our proof.  In Section \ref{pseudorandomness} we define the pseudorandomness notion satisfied by our ambient graph $\Gamma$ and prove that it is satisfied with high probability by a randomly chosen $\Gamma$ of sufficient density.  In Section \ref{model_theory}, we actually construct the infinitary setting we need and prove the necessary results to transfer results from this setting back to the finitary world; to the extent possible, all uses of model theory are kept to this section.  The reader who is willing to take for granted that the infinitary setting has the properties we need can ignore these two sections; the reader who is primarily interested in how we construct an infinitary setting can read just Section 3 and these two sections.

The proof in \cite{MR2815608} considers another pseudorandomness notion, PAIR.  In Section \ref{seminorms} we introduce our analog of PAIR, the generalized Gowers uniformity seminorms \cite{MR2373376,MR2920990} $||\cdot||_{U^{V,\mathcal{J}}_\infty(\mu^V_E)}$.  For technical reasons, we introduce a smaller family of seminorms in Subsection \ref{sec:principal_seminorms} and prove some basic properties, then introduce the full family in Subsection \ref{sec:nonprincipal_seminorms}.

It is well-known \cite{MR1054011} that PAIR and DISC are equivalent in the dense setting, and the key technical step of \cite{MR2815608} is \cite[Lemma 9a]{MR2815608}, showing that DISC implies PAIR in the sparse setting.  In our terminology, this becomes showing that $||\mathbb{E}(f\mid\mathcal{B}_{V,\mathcal{I}})||_{L^2(\mu^V_E)}=0$ implies $||f||_{U^{V,\mathcal{I}^\bot}_\infty(\mu^V_E)}=0$, which we refer to as the seminorm $||\cdot||_{U^{V,\mathcal{I}^\bot}_\infty(\mu^V_E)}$ being \emph{characteristic}.  We finally turn to this in Subsection \ref{sec:nonprincipal_characteristic} and Section \ref{random_measures}, where we prove this for successively more general classes of measures in several steps.

\subsection{Funding}

This work was supported by the National Science Foundation of the United States [DMS-1157580].

\subsection{Acknowledgements}

We are grateful to Ehud Hrushovski for providing the crucial motivating idea, and to Isaac Goldbring for many helpful discussions on this topic.

\section{Notation}\label{notation}

Throughout this paper we use a slightly unconventional notation for tuples which is particularly conducive to our arguments.  When $V$ is a finite set, a \emph{$V$-tuple from $G$} is a function $x_V:V\rightarrow G$.  If for each $v\in V$ we have designated an element $x_v\in G$, we write $x_V$ for the tuple $x_V(v)=x_v$.  Conversely, if we have specified a $V$-tuple $x_V$, we often write $x_v$ for $x_V(v)$.  When $V,W$ are disjoint sets, we write $x_V\cup x_W$ for the corresponding $V\cup W$-tuple.  (We will always assume $V$ and $W$ are disjoint when discussing $V\cup W$-tuples.)  When $I\subseteq V$ and $x_V$ is a given $V$-tuple, we write $x_I$ for the corresponding $I$-tuple: $x_I(i)=x_V(i)$ for $i\in I$.  We write $0^V$ for the tuple which is constantly equal to $0$.  (This is the only constant tuple we will explicitly refer to.)  When $B\subseteq M^{W\cup V}$, we will write $B(a_W)$ for the slice $\{x_V\mid a_W\cup x_V\in B\}$ corresponding to those coordinates.

\section{Families of Measures}\label{families}

To motivate our construction, we first consider the situation in large finite graphs.  Suppose we have a large finite set of vertices $G$ and a sparse random graph $\Gamma$ on $G$.  There are two natural measures we might consider on subsets of $G^2$: the usual normalized counting measure
\[\lambda(S)=\frac{|S|}{|G|^2}\]
and the counting measure normalized by $\Gamma$:
\[\lambda'(S)=\frac{|S\cap\Gamma|}{|\Gamma|}.\]
When we consider subsets of $G^3$, we have even more choices; we could normalize with respect to all possible triangles
\[\lambda_0(S)=\frac{|S|}{|G|^3},\]
or only those triangles entirely in $\Gamma$
\[\lambda_1(S)=\frac{|\{(x,y,z)\in S\mid (x,y)\in\Gamma, (x,z)\in\Gamma, (y,z)\in\Gamma\}|}{|\{(x,y,z)\mid(x,y)\in\Gamma, (x,z)\in\Gamma, (y,z)\in\Gamma\}|},\]
or only those triangles where certain specified edges belong to $\Gamma$:
\[\lambda_2(S)=\frac{|\{(x,y,z)\in S\mid (x,y)\in\Gamma, (x,z)\in\Gamma\}|}{|\{(x,y,z)\mid(x,y)\in\Gamma, (x,z)\in\Gamma\}|}.\]
Indeed, further consideration suggests that we have multiple choices for measures even on subsets of $G$: in addition to the normalized counting measure, we could fix any element $x\in G$ and define
\[\lambda_x(S)=\frac{|\{y\in S\mid (x,y)\in\Gamma\}|}{|\{y\mid (x,y)\in\Gamma\}|}.\]
When $\Gamma$ is a $k$-uniform hypergraph with $k>2$, we have yet more possibilities.

We therefore introduce a general notation for referring to all such measures.  We first describe this notation in the setting of a large finite graph, but we will primarily use it in the infinitary setting.  We assume that a value for $k$ and a $k$-uniform hypergraph $\Gamma$ on a set of vertices $G$ have been fixed.  When $V$ and $W$ are disjoint sets, $x_W\in G^W$, and $E$ is a $k$-uniform hypergraph on $V\cup W$, we define
\[\Gamma^V_{E,x_W}=\{x_V\in G^V\mid \forall e\in E\ x_e\in\Gamma\}.\]
Note the significance of our notation for tuples here: $x_e$ is a $k$-tuple which may consist both of elements from the fixed set $x_W$ and from $x_V$.  That is, $\Gamma^V_{E,x_W}$ is the collection of $x_V\in G^V$ such that map $x_{V\cup W}:V\cup W\rightarrow G$ is a homomorphism from $(V\cup W,E)$ to $\Gamma$.  For instance, in the case where $k=2$, so $\Gamma$ is a graph, $\Gamma^{\{1,2\}}_{\{(1,2)\},\emptyset}=\Gamma$, while $\Gamma^{\{1,2\}}_{\emptyset,\emptyset}=G^2$.

We then define
\[\mu^V_{E,x_W}(S)=\left\{\begin{array}{ll}
\frac{|S\cap\Gamma^V_{E,x_W}|}{|\Gamma^V_{E,x_W}|}&\text{if }|\Gamma^V_{E,x_W}|>0\\
0&\text{if }|\Gamma^V_{E,x_W}|=0
\end{array}\right..\]
For instance, in the measures above, $\lambda=\mu^{\{1,2\}}_{\emptyset,\emptyset}, \lambda'=\mu^{\{1,2\}}_{\{(1,2)\},\emptyset}, \lambda_0=\mu^{\{1,2,3\}}_{\emptyset,\emptyset}, \lambda_1=\mu^{\{1,2,3\}}_{\{(1,2),(1,3),(2,3)\},\emptyset}, \lambda_2=\mu^{\{1,2,3\}}_{\{(1,2),(1,3)\},\emptyset}$, and $\lambda_x=\mu^{\{1\}}_{\{(1,2)\},x}$.

In probabilistic terms, $\mu^V_{E,\emptyset}(S)$ is the probability that, if we choose a homomorphic copy of $E$ in $\Gamma$ uniformly at random, we obtain a copy belonging to $S$.  More generally $\mu^V_{E,x_W}(S)$ is the probability that if we extend $x_W$ uniformly at random to a copy of $E$ in $\Gamma$ that we obtain a copy belonging to $S$.

When $W$ and $x_W$ are clear from context, we just write $\mu^V_{E}$ for $\mu^V_{E,x_W}$ and $\Gamma^V_E$ for $\Gamma^V_{E,x_W}$, and call $x_W$ the \emph{background parameters} of $\mu^V_E$.

When integrating over $\mu^V_{E,x_W}$, we always assume the variable being integrated is $x_V$.

A key feature of this notation is that it makes it easy to specify the Fubini-type properties that we will eventually arrange for our measures to satisfy.  If $V=V_0\cup V_1$ where $V_0\cap V_1=\emptyset$ and $E'$ is the restriction of $E$ to vertices from $V_0\cup W$, we intend to have
\[\int\cdot\, d\mu^V_{E,x_W}=\iint\cdot\, d\mu^{V_1}_{E,x_{V_0}\cup x_W}d\mu^{V_0}_{E',x_W}.\]
To avoid having to endlessly specify the restriction of $E$ to the appropriate vertices, we will generally allow $E$ to have extra edges not included in the vertex set $V$; for instance, we will not distinguish between $\mu^{V_0}_{E',x_W}$ and $\mu^{V_0}_{E,x_W}$, and will usually write
\[\int\cdot\, d\mu^V_{E,x_W}(x_V)=\iint\cdot\, d\mu^{V_1}_{E,x_{V_0}\cup x_W}d\mu^{V_0}_{E,x_W},\]
even though $E$ is not a subset of ${{V_0\cup W}\choose k}$.

\section{\texorpdfstring{$\sigma$}{sigma}-Algebras}\label{sigma_algebras}

Most of our work will be carried out in the setting of uncountable hypergraphs with probability measures.  We face the following difficulty, even in the graph case: suppose we are working with the graph $(M,E)$ and have a $\sigma$-algebra $\mathcal{B}_1$ on $M$.  Then $E\subseteq M^2$, but it need not be the case that $E$ is measurable with respect to the product algebra $\mathcal{B}_1\times\mathcal{B}_1$.  Since $\mathcal{B}_1\times\mathcal{B}_1$ is generated by rectangles, measurability of $E$ with respect to $\mathcal{B}_1\times\mathcal{B}_1$ is actually a strong \emph{combinatorial} requirement on $E$---indeed, as we will see below, or as noted in \cite{goldbring:_approx_logic_measure}, closely related to the properties of regularity partitions for $E$.

Our solution draws from Keisler's notion of a graded probability space \cite{MR0491155}: we need to work with $\sigma$-algebras $\mathcal{B}_n$ on $n$-tuples for every $n$ so that $\mathcal{B}_m\times\mathcal{B}_n\subseteq\mathcal{B}_{m+n}$, but we allow $\mathcal{B}_{m+n}$ to contain additional measurable sets beyond those required by the product.  To better match our tuple notation, we will actually work with $\sigma$-algebras on $V$-tuples for all finite $V$ (though we will ultimately define them to depend only on $|V|$, and not on the particular elements of the set $V$).  

We need certain sub-$\sigma$-algebras giving those sets measurable in certain well-defined ways.  For instance, we wish to define generalizations of product algebras like $\mathcal{B}_1\times\mathcal{B}_1$.

\begin{definition}
\label{def:bnk_algebra}

Suppose that for every finite set of indices $V$ we have a Boolean algebra $\mathcal{B}^0_V$ on subsets of $M^V$ such that:
\begin{itemize}
\item $\emptyset\in\mathcal{B}^0_V$ and $M^V\in\mathcal{B}^0_V$,
\item $\mathcal{B}^0_V\times\mathcal{B}^0_W\subseteq\mathcal{B}^0_{V\cup W}$,
\item Whenever $a_W\in M^W$ and $B\in\mathcal{B}^0_{V\cup W}$, the projection $B(a_W)\in M^V$.
\end{itemize}

For $I\subseteq V$, we define $\mathcal{B}^0_{V,I}$ to be the Boolean algebra generated by subsets of $M^n$ of the form
\[\{x_V\in M^V\mid x_I\in B\}\]
where $B\in\mathcal{B}^0_I$.

If $\mathcal{I}\subseteq \mathcal{P}(V)$ then we write $\mathcal{B}^0_{V,\mathcal{I}}$ for the Boolean algebra generated by $\bigcup_{I\in\mathcal{I}}\mathcal{B}^0_{V,\mathcal{I}}$.  When $k\leq|V|$, we define $\mathcal{B}^0_{V,k}$ to be the Boolean algebra $\mathcal{B}^0_{V,\{I\subseteq V\mid |I|=k\}}$.  

For any $I\subseteq V$, we write ${<}I$ for the set of proper subsets of $I$.  The \emph{principal} algebras are those of the form $\mathcal{B}^0_{V,{<}V}=\mathcal{B}^0_{V,|V|-1}$.

In all cases, we drop the superscript $^0$ to indicate the $\sigma$-algebra generated by the algebra.
\end{definition}

Throughout this paper our primary example of such a system of algebras will be the for $\mathcal{B}^0_V$ to be the collection fo sets of $V$-tuples definable in a model $\mathfrak{M}$ using parameters.

The algebras $\mathcal{B}^0_{V,\mathcal{I}}$ are generally uncountable, and so the corresponding $\sigma$-algebras $\mathcal{B}_{V,\mathcal{I}}$ are generally non-separable.  (It is possible to recover separability by allowing only formulas whose parameters come from an elementary submodel.  This causes some additional complications, since the slices of some set $A\subseteq M^2$ are no longer necessarily measurable; rather, the slices are measurable with respect to some slightly larger $\sigma$-algebra which depends on the choice of slice.  These complications can be addressed by a small amount of additional model-theoretic work; this separable approach is used in \cite{towsner:modeltheoretic,goldbring:_approx_logic_measure}.)  These $\sigma$-algebras are closely related to the Szemer\'edi Regularity Lemma; for instance, in \cite{goldbring:_approx_logic_measure} it is shown that the usual regularity lemma follows almost immediately from the existence of the projection of a set onto $\mathcal{B}_{\{1,2\},1}$.

Note that while a $\sigma$-algebra is well-defined independently of the choice of a particular measure, notions like the projection onto a $\sigma$-algebra do depend on a particular choice of measure.


The first introduction of these algebras that we know of is in \cite{tao08Norm}, where Tao already notes the relationship with the Gowers uniformity norms which we will discuss in detail below.  The work in this paper builds on further developments in \cite{towsner:MR2529651,tao07}.

There is some flexibility in the choice of the set $\mathcal{I}$; for instance, $\mathcal{B}_{\{1,2,3\},\{\{1,2\}\}}=\mathcal{B}_{\{1,2,3\},\{\{1,2\},\{1\}\}}$ (since $\{1,2\}\in\mathcal{I}$, we already have sets depending only on the coordinate $1$, so adding $\{1\}$ does nothing).  This leads to two canonical choices for $\mathcal{I}$: a minimal choice with only the sets of coordinates absolutely necessary, or a maximal choice which adds every set of coordinates allowed without changing the meaning.  Depending on the situation, we want one or the other canonical form.

\begin{lemma}
  If for every $I\in\mathcal{I}$ there is an $I'\in\mathcal{I}'$ with $I\subseteq I'$ then $\mathcal{B}_{V,\mathcal{I}}\subseteq\mathcal{B}_{V,\mathcal{I}'}$.
\end{lemma}
\begin{proof}
  It suffices to show that if $I\subseteq I'$ then $\mathcal{B}^0_{V,I}\subseteq\mathcal{B}^0_{V,I'}$.  But this is easily seen from the definition, since if $B\in\mathcal{B}^0_{I}$, $B\times M^{I'\setminus I}\in\mathcal{B}^0_{I'}$, and therefore
\[\{x_V\in M^V\mid x_I\in B\}=\{x_V\in M^V\mid x_{I'}\in B\times M^{I'\setminus I}\}\in\mathcal{B}^0_{V,I'}.\]
\end{proof}

\begin{cor}
  For any $V,\mathcal{I}$, there exist $\mathcal{I}_0,\mathcal{I}_1$ such that:
  \begin{enumerate}
  \item $\mathcal{B}_{V,\mathcal{I}}=\mathcal{B}_{V,\mathcal{I}_0}=\mathcal{B}_{V,\mathcal{I}_1}$,
  \item $\mathcal{I}_0$ is downwards closed: if $I\in\mathcal{I}_0$ and $J\subseteq I$ then $J\in\mathcal{I}_0$,
  \item If $I,J\in\mathcal{I}_1$ then $J\not\subseteq I$.
  \end{enumerate}
\end{cor}

\begin{definition}
  Given $\mathcal{I},\mathcal{J}\subseteq\mathcal{P}(V)$, we define $\mathcal{I}\wedge\mathcal{J}$ to consist of those sets $K$ such that there is an $I\in\mathcal{I}$ and a $J\in\mathcal{J}$ such that $K\subseteq I\cap J$.
\end{definition}
We often equate $J$ with $\{J\}$, so $\mathcal{I}\wedge J=\mathcal{I}\wedge\{J\}$.

\section{Hypergraph Removal}\label{hypergraph_removal}

In this section, we present a proof of the ordinary hypergraph removal theorem, essentially the one given in \cite{towsner:modeltheoretic}, which is in turn based on the arguments in \cite{tao07,tao:MR2259060}.  We first state a necessary property on measures, and prove a lemma reminiscent of the hypergraph counting lemma.

\begin{definition}
  Let $\nu^V$ be a probability measure on $\mathcal{B}_{V}$.  We say $\nu^V$ \emph{has $J$-regularity} for $J\subseteq V$ if for any $\mathcal{I}\subseteq\mathcal{P}(V)$:
  \begin{quote}
    For any $\{f_I\}_{I\in\mathcal{I}}$ and $g\in L^\infty(\mathcal{B}_{V,J})$ such that for each $I\in\mathcal{I}$, $f_I\in L^\infty(\mathcal{B}_{V,I})$,
\[\int (g-\mathbb{E}(g\mid\mathcal{B}_{V,\mathcal{I}\wedge J}))\prod_{I\in\mathcal{I}}f_I\, d\nu^V=0.\]
  \end{quote}
\end{definition}

Note that if there is any $I\in\mathcal{I}$ with $J\subseteq I$, so $J\in\mathcal{I}\wedge J$, then when $g\in L^\infty(\mathcal{B}_{V,J})$, we have $g=\mathbb{E}(g\mid\mathcal{B}_{V,\mathcal{I}\wedge J})$, and therefore the statement is trivial.  So we are usually concerned with the case where $J\not\subseteq I$ for all $I\in\mathcal{I}$.

Given a measure $\nu^V$ on $\mathcal{B}_V$, for any $J\subseteq V$ there is a natural measure $\nu^J$ on $\mathcal{B}_J$: there is a canonical embedding of $\mathcal{B}_J$ as $\mathcal{B}_{V,J}$, and we take $\nu^J(B)=\nu^V(J\times X^{V\setminus J})$.  A basic property we expect of $\nu^V$ is that when $B\in\mathcal{B}_J$ and $C\in\mathcal{B}_{V\setminus J}$ then $\nu^V(B\times C)=\nu^J(B)\nu^{V\setminus J}(C)$.  We say $\nu^V$ \emph{extends the product} $\nu^J\times\nu^{V\setminus J}$ in this case.
\begin{lemma}
If $\nu^V$ has $J$-regularity then $\nu^V$ extends the product $\nu^J\times\nu^{V\setminus J}$.  When $|J|=1$ and $\nu^V$ extends the product $\nu^J\times\nu^{V\setminus J}$, $\nu^V$ has $J$-regularity.
\end{lemma}
\begin{proof}
Let $\nu^V$ with $J$-regularity be given.  It suffices to show that for any $B\in\mathcal{B}_{V,J}$ and $C\in\mathcal{B}_{V,V\setminus J}$, $\nu^V(B\times C)=\nu^j(B)\nu^{V\setminus J}(C)$.  Let $I=V\setminus J$, $\mathcal{I}=\{I\}$, and take any such $B$ and $C$.  Note that $\mathcal{B}_{V,\mathcal{I}\wedge J}=\mathcal{B}_{V,\emptyset}$, which is the trivial $\sigma$-algebra.  In particular, for any $g$, $\mathbb{E}(g\mid\mathcal{B}_{V,\emptyset})$ is the function constantly equal to $\int gd\nu^V$.

By $J$-regularity,
\begin{align*}
  0
&=\int (\chi_B-\mathbb{E}(\chi_B\mid\mathcal{B}_{V,\mathcal{I}\wedge\{j\}}))\chi_Cd\nu^V\\
&=\int \chi_B\chi_C d\nu^V-\int\mathbb{E}(\chi_B\mid\mathcal{B}_{V,\mathcal{I}\wedge\{j\}})\chi_Cd\nu^V\\
&=\int \chi_B\chi_C d\nu^V-\int \nu^j(B)\chi_Cd\nu^V\\
&=\int \chi_B\chi_C d\nu^V-\nu^j(B)\nu^{V\setminus\{j\}}(C).
\end{align*}

Suppose $J=\{j\}$ and $\nu^V$ extends $\nu^J\times\nu^{V\setminus J}$.  To show $J$-regularity, consider some $\mathcal{I}\subseteq\mathcal{P}(V)$ so that for each $I\in\mathcal{I}$, $I\cap J\subsetneq J$---that is, $j\not\in I$.  If for each $I\in\mathcal{I}$ we have $f_I\in L^\infty(\mathcal{B}_{V,I})$ then we have $\prod_I f_I\in L^\infty(\mathcal{B}_{V,V\setminus J})$.  Then for any $g\in L^\infty(\mathcal{B}_{V,\{j\}})$ we have
\begin{align*}
  \int (g-\mathbb{E}(g\mid\mathcal{B}_{V,\mathcal{I}\wedge J}))\prod_{I\in\mathcal{I}}f_I\, d\nu^V
&=\int (g-\mathbb{E}(g\mid\mathcal{B}_{V,\mathcal{I}\wedge J}))d\nu^{j}\int\prod_{I\in\mathcal{I}}f_I\, d\nu^{V\setminus \{j\}}\\
&=0\cdot\int\prod_{I\in\mathcal{I}}f_I\, d\nu^{V\setminus \{j\}}\\
&=0.
\end{align*}

\end{proof}

The next theorem is an infinitary analog of hypergraph removal.  
\begin{theorem}\label{counting}
Suppose $\nu^V$ has $J$-regularity for all $J\subseteq V$ with $|J|\leq k$, that $k< |V|$, $\mathcal{I}\subseteq{V\choose k}\cup\{V\}$, and for each $I\in\mathcal{I}$ we have a set $A_I\in\mathcal{B}_{I}$ such that $A_V\in\mathcal{B}_{V,<k}$.  Further, suppose there is a $\delta>0$ such that whenever $B_I\in\mathcal{B}_{I}^0$, $\nu^I(A_I\setminus B_I)<\delta$ for all $I\in\mathcal{I}$, and $B_V\in\mathcal{B}^0_{V,<k}$, $\bigcap_{I\in\mathcal{I}} B_I$ is non-empty.  Then $\nu^V(\bigcap_{I\in\mathcal{I}}A_I)>0$.
\end{theorem}
We usually apply this with $V\not\in\mathcal{I}$, but have to deal with with a more general term to make the induction go through.
\begin{proof}
We proceed by main induction on $k$.  When $k=1$, the claim is simple: if there is any $I_0$ with $\nu^{I_0}(A_{I_0})<\delta$, we could take $B_{I_0}=\emptyset$ and $B_I=M^I$ for $I\neq I_0$, contradicting the assumption.  So $\nu^V(\bigcap A_I)=\nu^V(\bigcap A_I)=\prod\nu^I(A_I)\geq\delta^{|V|+1}>0$ by the previous lemma.

So we assume that $k>1$ and that whenever $B_I\in\mathcal{B}^0_{I}$ and $\nu^I(A_I\setminus B_I)<\delta$ for all $I$, $\bigcap_{I\in\mathcal{I}} B_I$ is non-empty.  Throughout this proof, the variables $I$ and $I_0$ range over elements of $\mathcal{I}$.  We first show that, without loss of generality, we may assume each $A_{I}$ belongs to $\mathcal{B}_{I,{<}I}$, by showing that for each $I_0\in\mathcal{I}\setminus\{V\}$, there is some set $A'_{I_0}\in\mathcal{B}_{I_0,{<}I_0}$ with the property that, if we replace $A_{I_0}$ by $A'_{I_0}$, the assumptions of the theorem all hold, and such that if we show the conclusion for the modified family of sets, the conclusion also holds for the original family.

\begin{claim}
For any $I_0$, there is an $A'_{I_0}\in\mathcal{B}_{I_0,{<}I_0}$ such that:
\begin{itemize}
  \item whenever $B_I\in\mathcal{B}^0_{I}$ for each $I$, $\nu^I(A_I\setminus B_I)<\delta$ for each $I\neq I_0$, and $\nu^{I_0}(A'_{I_0}\setminus B_{I_0})<\delta$, $\bigcap_{I\in\mathcal{I}}B_I$ is non-empty, and
  \item if $\nu^V(A'_{I_0}\cap \bigcap_{I\neq I_0}A_I)>0$, $\nu^V(\bigcap_{I\in\mathcal{I}} A_I)>0$.
\end{itemize}
\end{claim}
\begin{claimproof}
 Define $A'_{I_0}:=\{x_{I_0}\mid \mathbb{E}(\chi_{A_{I_0}}\mid\mathcal{B}_{I_0,<I_0})(x_{I_0})>0\}$.  If $\nu^V(A'_{I_0}\cap\bigcap_{I\neq I_0}A_I)>0$ then we have
 \[\int \mathbb{E}(\chi_{A_{I_0}}\mid\mathcal{B}_{I_0,<I_0})\prod_{I\neq I_0}\chi_{A_I}\,d\nu^V>0\]
 and since $\nu^V$ has $I_0$-regularity, this implies that $\nu^V(\bigcap A_I)>0$.
 
 Suppose that for each $I$, $B_I\in\mathcal{B}^0_{I}$ with $\nu^I(A_I\setminus B_I)<\delta$ for $I\neq I_0$ and $\nu^{I_0}(A'_{I_0}\setminus B_{I_0})<\delta$.  Since
\[\nu^{I_0}(A_{I_0}\setminus A'_{I_0})=\int\chi_{A_{I_0}}(1-\chi_{A'_{I_0}})\,d\nu^{I_0}=\int \mathbb{E}(\chi_{A_{I_0}}\mid\mathcal{B}_{I_0,<I_0})(1-\chi_{A'_{I_0}})\,d\nu^{I_0}=0,\]
we have $\nu^{I_0}(A_{I_0}\setminus B_{I_0})<\delta$ as well, and therefore $\bigcap_{I\in\mathcal{I}} B_I$ is non-empty.
\end{claimproof}

Fix finitely many sets from $\mathcal{B}^0_{[1,k],k-1}$ and let $\mathcal{B}$ be the $\sigma$-algebra generated by these sets together with $\mathcal{B}^0_{[1,k],k-2}$.  By abuse of notation, we treat $\mathcal{B}$ as a sub-$\sigma$-algebra of every $\mathcal{B}_I$.  By choosing enough sets, we may ensure that for every $I$, $\|\chi_{A_I}-\mathbb{E}(\chi_{A_I}\mid\mathcal{B})\|_{L^2(\nu^I)}<\frac{\sqrt{\delta}}{\sqrt{2}(|\mathcal{I}|+2)}$ for each $I$.  For each $I$, set $A^*_I=\{a_I\mid \mathbb{E}(\chi_{A_I}\mid\mathcal{B})(a_I)>\frac{|\mathcal{I}|}{|\mathcal{I}|+1}\}$.

\begin{claim}
For each $I$, $\nu^I(A_I\setminus A^*_I)\leq \delta/2$.
\end{claim}
\begin{claimproof}
$A_I\setminus A^*_I$ is the set of points such that $\left(\chi_{A_I}-\mathbb{E}(\chi_{A_I}\mid\mathcal{B})\right)(\vec a)\geq\frac{1}{|\mathcal{I}|+1}$.  By Chebyshev's inequality, the measure of this set is at most
\[(|\mathcal{I}|+1)^2\int (\chi_{A_I}-\mathbb{E}(\chi_{A_I}\mid\mathcal{B}))^2\,d\nu^I=(|\mathcal{I}|+1)^2\|\chi_{A_I}-\mathbb{E}(\chi_{A_I}\mid\mathcal{B})\|_{L^2(\nu^I)}^2\leq\frac{\delta}{2}.\]
\end{claimproof}

\begin{claim}
$\nu^V(\bigcap_I A_I)\geq\nu^V(\bigcap_I A^*_I)/\left(|\mathcal{I}|+1\right)$. 
\end{claim}
\begin{claimproof}
 For each $I_0$, 
 \begin{align*}
 \nu^V((A^*_{I_0}\setminus A_{I_0})\cap \bigcap_{I\neq I_0} A^*_I)
 &=\int \chi_{A^*_{I_0}}(1-\chi_{A_{I_0}})\prod_{I\neq I_0}\chi_{A^*_I}\,d\nu^V\\
 &=\int\chi_{A^*_{I_0}}(1-\mathbb{E}(\chi_{A_{I_0}}\mid\mathcal{B}))\prod_{I\neq I_0}\chi_{A^*_I}\,d\nu^V\\
 &\leq \frac{1}{|\mathcal{I}|+2}\int \prod_{I\in\mathcal{I}}\chi_{A^*_I}\,d\nu^V\\
 &=\frac{1}{|\mathcal{I}|+2}\nu^V(\bigcap_{I\in\mathcal{I}} A^*_I).
\end{align*}

But then
\begin{align*}
\nu^V(\bigcap_{I\in\mathcal{I}} A^*_I\setminus (\bigcap_{I\in\mathcal{I}} A_I))
&\leq \sum_{I_0}\nu^V((A^*_{I_0}\setminus A_{I_0})\cap \bigcap_{I\neq I_0} A^*_I)\\
&\leq \frac{|\mathcal{I}|}{|\mathcal{I}|+1}\nu^V(\bigcap_{I\in\mathcal{I}} A^*_I).
\end{align*}
\end{claimproof}

Each $A^*_I$ may be written in the form $\bigcup_{i\leq r_I}A^*_{I,i}$ where $A^*_{I,i}=\bigcap_{J\in {I\choose k-1}\cup V}A^*_{I,i,J}$, $A^*_{I,i,J}$ is an element of $\mathcal{B}^0_{V,J}$, and $A^*_{I,i,V}\in\mathcal{B}_{I,k-2}$.  We may assume that if $i\neq i'$ then $A^*_{I,i}\cap A^*_{I,i'}=\emptyset$.


We have
\[\nu^V(\bigcap_I A^*_I)=\nu^V(\bigcup_{\vec i\in\prod_I [1,r_I]}\bigcap_I\bigcap_{J\in{I\choose k-1}\cup\{V\}}A^*_{I,i_I,J}).\]
For each $\vec i\in\prod_I [1,r_I]$, let $D_{\vec i}=\bigcap_I\bigcap_{J\in{I\choose k-1}\cup\{V\}}A^*_{I,i_I,J}$.  Each $A^*_{I,i_I,J}$ is an element of $\mathcal{B}_{V,J}$, so we may group the components and write $D_{\vec i}=\bigcap_{J\in{V\choose k-1}\cup \{V\}}D_{\vec i,J}$ where $D_{\vec i,V}=\bigcap_I A^*_{I,i_I,V}$ and otherwise $D_{\vec i,J}=\bigcap_{I\supset J}A^*_{I,i_I,J}$.

Suppose, for a contradiction, that $\nu^V(\bigcap_I A^*_I)=0$.  Then for every $\vec i\in\prod_I [1,r_I]$, $\nu^V(D_{\vec i})=\nu^V(\bigcap_J D_{\vec i,J})=0$.  By the contrapositive of the inductive hypothesis, for each $\gamma>0$, there is a collection $B_{\vec i,J}\in\mathcal{B}^0_{V,J}$ and $B_{\vec i,V}\in \mathcal{B}^0_{V,k-2}$ such that $\nu^V(D_{\vec i,J}\setminus B_{\vec i,J})<\gamma$ and $\bigcap_J B_{\vec i,J}=\emptyset$.  In particular, this holds with $\gamma=\frac{\delta}{6({k\choose k-1}+1)(\prod_I r_I)(\max_I r_I)}$.

For $J\neq V$ and $I\supset J$, define
\[B^*_{I,i,J}=A^*_{I,i,J}\cap\bigcap_{\vec i,i_I=i}\left[B_{\vec i,J}\cup\bigcup_{I'\neq I, I'\supset J}\overline{A^*_{I',i_{I'},J}}\right].\]
For each $I$, let $A^0_{I,i,V}$ be an element of $\mathcal{B}^0_{V,k-2}$ with $\nu^I(A^*_{I,i,V}\bigtriangleup A^0_{I,i,V})<\frac{\delta}{6(|\mathcal{I}|+\prod_I r_I+1)}$.  We analogously define
\[B^{**}_{I,i,V}=A^0_{I,i,V}\cap\bigcap_{\vec i,i_I=i}\left[B_{\vec i,V}\cup\bigcup_{I'\neq I}\overline{A^0_{I',i_{I'},V}}\right]\]
and then set
\[B^*_{I,i,V}=\{x_I\mid \nu^{V\setminus I}(\{x_{V\setminus I}\mid (x_I,x_{V\setminus I})\in B^{**}_{I,i,V})\geq \frac{1}{|\mathcal{I}|+\prod_I r_I+1}\}.\]

Finally we set
\[B^*_I=\bigcup_{i\leq r_I}\bigcap_J B^*_{I,i,J}.\]
Note that $B^*_I\in\mathcal{B}^0_I$ and $B^*_V\in\mathcal{B}^0_{V,<k}$.  Note that $B^{**}_{V,i,V}=B^*_{V,i,V}$.

\begin{claim}
$\nu^{I_0}(A^*_{I_0}\setminus B^*_{I_0})\leq\delta/2$.
\end{claim}
\begin{claimproof} 
Let $\theta=\frac{1}{|\mathcal{I}|+\prod_I r_I+1}$.  If $S\subseteq M^V$, define
\[\pi_{I_0,}(S)=\{x_{I_0}\mid \nu^{V\setminus I_0}(\{x_{V\setminus I_0}\mid x_V\in S\})\geq\theta\}.\]

We will show that
\begin{align*}
A^*_{I_0}\setminus B^*_{I_0}
&\subseteq \bigcup_{\vec i,J\subset I}(D_{\vec i,J}\setminus B_{\vec i,J})\\
&\ \ \cup\bigcup_{I,i}\pi_{I_0}(A^*_{I,i,V}\bigtriangleup A^0_{I,i,V})\\
&\ \ \cup\bigcup_{\vec i}\pi_{I_0}(D_{\vec i,V}\setminus B_{\vec i,V})
\end{align*}
which suffices to give the claim.

Observe that if $x_{I_0}\in A^*_{I_0}\setminus B^*_{I_0}$ then $x_{I_0}\in\bigcap_J A^*_{I_0,i,J}$ for some $i$.  Then there must be some $J$ so $x_{I_0}\not\in B^*_{I,i,J}$.  If $J\neq V$ then $x\not\in B_{\vec i,J}$ and $x_{I_0}\in\bigcap_{I'\supset J}A^*_{I',i_{I'},J}$, so $x_{I_0}\in D_{\vec i,J}\setminus B_{\vec i,J}$.

If $J=V$, suppose $x_{I_0}\in A^*_{I_0,i,V}\setminus B^*_{I_0,i,V}$ but
\[x_{I_0}\not\in \bigcup_{I}\pi_{I_0,\delta}(A^*_{I,i,V}\bigtriangleup A^0_{I,i,V})\cup\bigcup_{\vec i}\pi_{I_0,\delta}(D_{\vec i,V}\setminus B_{\vec i,V}).\]
Then there must be some $x_{V\setminus I_0}$ so that $x_V\not\in B^{**}_{I_0,i,V}\cup\bigcup_{I,i}(A^*_{I,i,V}\bigtriangleup A^0_{I,i,V})\cup\bigcup_{\vec i}(D_{\vec i,V}\setminus B_{\vec i,V})$, which is a contradiction by the same argument as in the $J\neq V$ case.
\end{claimproof} 

Since $\nu^I(A_I\setminus A^*_I)\leq\delta/2$, it follows that $\nu^I(A_I\setminus B^*_I)\leq\delta$.  Therefore $\bigcap_I B^*_I\neq\emptyset$.

\begin{claim}
  \[\bigcap_I B^*_I\subseteq \bigcup_{\vec i}\bigcap_J B_{\vec i,J}.\]
\end{claim}
\begin{proof}
  Suppose $x\in\bigcap_I B^*_I$.  Then for each $I$, there is an $i_I$ so that $x\in \bigcap_J B^*_{I,i,J}$.  Therefore $x\in A^*_{I,i_I,J}$ for $J\neq V$ and $x\in A^0_{I,i_I,V}$.  Therefore $x\in B_{\vec i,J}$ for each $J$.
\end{proof}

Since $\bigcap_{I}B^*_{I}$ is non-empty, there is some $\vec i$ such that $\bigcap_J B_{\vec i,J}\neq\emptyset$.  But this leads to a contradiction, so it must be that $\nu^V(\bigcap_I A^*_I)>0$, and therefore, as we have shown, $\nu^V(\bigcap_{I\in\mathcal{I}} A_I)\geq\frac{1}{|\mathcal{I}|+1}\nu^V(\bigcap_{I\in\mathcal{I}} A^*_I)>0$.
\end{proof}

In order to prove the hypergraph removal theorem, we would then hope to argue as follows: the failure of hypergraph removal implies the existence of a family of counterexamples of unbounded size.  We could then use a bit of model theory---the ultraproduct construction---to obtain an infinite hypergraph together with some measures in which $\nu^V(\bigcap_{I\in{V\choose k}}A_I)=0$ for a family of sets $A_I$ corresponding to the graph we are trying to remove.  By the previous theorem, we would have an arbitrarily small family of definable sets $B_I$, and we would then argue that that these sets correspond to sets in the finite models whose removal causes the removal of all copies of the hypergraph.  The only remaining difficulty in this argument is showing that the measure we obtain has $J$-regularity for all $J\subseteq V$.

In the remainder of this section, we carry out the proof for the dense case of hypergraph removal; this will necessarily involve some model theory.

\begin{lemma}\label{product_regular}
  Suppose that for each $J\subseteq V$, $\nu^J$ is a probability measure on $\mathcal{B}_J$ such that 
  \begin{itemize}
  \item For any $B\in\mathcal{B}_V$, the function $x_{V\setminus J}\mapsto\nu^J(B(x_{V\setminus J}))$ is measurable with respect to $\mathcal{B}_{V\setminus J}$, and
  \item For any $L^\infty(\nu^V)$ function $f$, $\int f\, d\nu^V=\iint f\, d\nu^Jd\nu^{V\setminus J}$.
  \end{itemize}
  Then $\nu^V$ has $J$-regularity for every $J\subseteq V$.  
\end{lemma}
\begin{proof}
We have
\begin{align*}
 &\int (g-\mathbb{E}(g\mid\mathcal{B}_{V,\mathcal{I}\wedge J}))\prod_{I}f_I\, d\nu^V\\
= &\int (g-\mathbb{E}(g\mid\mathcal{B}_{V,\mathcal{I}\wedge J}))\prod_{I}f_I\, d\nu^{J}d\nu^{V\setminus J}.\\
\end{align*}

For each $a_{V\setminus J}$, the function $\prod_{I}f_I(a_{I\setminus J},x_{I\cap J})$ is measurable with respect to $\mathcal{B}_{V,<J}$, so we have
\[\int (g-\mathbb{E}(g\mid\mathcal{B}_{V,\mathcal{I}\wedge J}))\prod_{I}f_I\, d\nu^{J}=0.\]
Since this holds for every $a_{V\setminus J}$, the claim follows by integrating over all choices of $a_{V\setminus J}$.
\end{proof}

 \begin{definition}
 Let $K,A$ be $k$-uniform hypergraphs on vertex sets $V(K),V(A)$ respectively.  $\pi:V(K)\rightarrow V(A)$ is a \emph{homomorphism} if whenever $e\in K$, $\pi"e\in A$.  (That is, $\pi$ maps edges to edges.)  $hom(K,A)$ is the number of distinct homomorphisms from $K$ to $A$.  If $K,A$ are $k$-uniform hypergraphs, we write
 \[d(K,A)=\frac{hom(K,A)}{|V(A)|^{|V(K)|}}.\]
 \end{definition}

\begin{theorem}[Hypergraph Removal]\label{dense_hypergraph_removal}
  For every $k$-uniform hypergraph $K$ and constant $\epsilon>0$, there is a $\delta$ such that whenever $A$ is a finite $k$-uniform hypergraph with $d(K,A)<\delta$, there is a subset $L$ of $A$ with $|L|\leq\epsilon{|V(A)|\choose k}$ such that $hom(K,A\setminus L)=0$.
\end{theorem}
\begin{proof}
  Suppose not.  Let $K,\epsilon$ be a counterexample, and since there is no such $\delta$, for each $n$ we may choose a $k$-uniform hypergraph $A^n$ with $d(K,A^n)<1/n$ such that there is no such subset $L$ of $A^n$.  Clearly $|V(A^n)|\rightarrow\infty$.  We view each $A^n$ as a model, with $M^n=V(A^n)$ the set of points, $A^n$ a $k$-ary relation on $V(A^n)$, and predicates making the normalized counting measure $\nu^J_n$ on $V(A^n)^J$ definable for each $J\subseteq V(K)$.  In particular, this means the counting measure is a \emph{uniformly} definable Keisler probability measure.

Let $V=V(K)$.  For each $I\in K$, let $A^n_I=\{x_V\mid x_I\in A^n\}$.  Note that the homomorphisms from $K$ to $A^n$ consist exactly of the elements of 
\[\bigcap_{I\in K}A^n_I,\]
and therefore $d(K,A^n)=\nu^V_n(\bigcap_{I\in K}A^n_I)$.  In particular, we have $\nu_n^V(\bigcap_{I\in K}A^n_I)\rightarrow 0$.

Now take an ultraproduct of the models $(M^n,A^n,\ldots)$ to obtain $\mathfrak{M}=(M,A,\ldots)$.  (See \cite{MR1643950} for the construction and, in particular, the demonstration that the measures defined by $\nu^J$, the ultraproduct of the $\nu^J_n$, extend to probability measures on $\mathcal{B}_J$.)  By \cite{MR0491155,MR732752} (or see Section \ref{model_theory}), the conditions in the statement of Lemma \ref{product_regular} hold in $\mathfrak{M}$, and therefore $\nu^V$ has $I$-regularity for all $I\subseteq V$.  We have $\nu^V(\bigcap_{I\in K}A_I)=0$, and therefore by the previous theorem, there are $B_I\in\mathcal{B}^0_{V,I}$ such that $\nu^V(E\setminus B_I)<\frac{\epsilon}{|K|}$ and $\bigcap_{I\in K}B_I=\emptyset$.  Let $C=\bigcup_I (A_I\setminus B_I)$, so $\nu^{[1,k]}(L)<\epsilon$.  $L$ is definable from parameters in $M$, and therefore
\[\bigcap_{I\in K}(A_I\setminus L)=\emptyset\]
is a formula, which is therefore satisfied by the corresponding set in almost every $(M^n,A^n,\ldots)$.  Let $L^n$ be the set defined in the model $(M^n,A^n,\ldots)$ by the formula defining $L$.  Then there is some sufficiently large $n$ such that $(\nu^n)^V(L_n)<\epsilon$ but $\bigcap_{I\in K}(A^n_I\setminus L_n)=\emptyset$, contradicting the assumption.
\end{proof}

Note that this argument, essentially unchanged, also gives variants like directed removal \cite{alon:MR2087940} (note that we never require the sets $A_I$ to be symmetric) or removal of colored graphs \cite{austin:MR2666763} (take the $A_I$ to be any of several sets, one corresponding to each color).


Our goal is to obtain the same result when $A$ is not a dense hypergraph, but rather a dense subset of a sparse random graph.  The main idea is that we will replace $\nu^V$ with a measure concentrating on the sparse pseudorandom graph; however this will not satisfy the easy Fubini decomposition we used for the dense case, so we will need to use the pseudorandomness---plus a large amount of additional machinery---to prove that the resulting measures nonetheless have regularity.


\section{Pseudorandomness and the Fubini Property}\label{pseudorandomness}

In this section we examine the Fubini property of measures more carefully as a property of pseudorandom hypergraphs.  The Fubini properties tell us that different methods of counting homomorphic copies of $E$ give the same values.  For instance $\int f(x_V) d\mu^V_{E,\emptyset}$ is the expected value if we choose a copy $x_V$ of $E$ at random and evaluate $f(x_V)$.  On the other hand $\iint f(x_V) d\mu^{V_0}_{E,x_{V_1}}d\mu^{V_1}_{E,\emptyset}$ is the expected value of the process where we first choose a copy of $(V_1,E\upharpoonright [V_1]^k)$ at random, and then extend this copy to a copy of $E$ at random and evaluate $f$ on the result.

For a simple example where these processes differ, consider the graph with vertex set $G_0\cup G_1$ (with $G_0\cap G_1=\emptyset$) where $|G_0|=2n$, $|G_1|=n^{2/3}$, and take $\Gamma$ to be the graph whose edges consist of a matching on $G_0$ (that is, exactly $n$ edges with each vertex in $G_0$ an endpoint of exactly one of them) and all possible edges on $G_1$.  Then when $n$ is large, almost all edges of $\Gamma$ belong to the complete subgraph $G_1$ while almost all vertices belong to $G_0$.  Let $V=\{0,1\}$ and $E$ be just the edge connecting $0$ to $1$.  Then $\mu^V_{E,\emptyset}$ simply counts edges in $\Gamma$, and so $\int \chi_{G_0\times G_0}(x_V)d\mu^V_{E,\emptyset}$ is the fraction of edges contained in $G_0$---namely, almost none of them, so $\int \chi_{G_0\times G_0}(x_V)d\mu^V_{E,\emptyset}\rightarrow 0$ as $n\rightarrow \infty$.  On the other hand, taking $V_0=\{0\}$ and $V_1=\{1\}$, $\iint \chi_{G_0\times G_0}(x_V)d\mu^{V_0}_{E,x_{V_1}}d\mu^{V_1}_{E,\emptyset}$ is the average where we first select a vertex $x_{V_1}$---which, with high probability, belongs to $G_0$---and then choose a second vertex $x_{V_0}$ from among those vertices connected to $x_{V_1}$ (that is, we only consider those extensions which actually give copies of $E$); when $x_{V_1}\in G_0$, $x_{V_0}$ is, with high probability, its matched element, so $\iint \chi_{G_0\times G_0}(x_V)d\mu^{V_0}_{E,x_{V_1}}d\mu^{V_1}_{E,\emptyset}\rightarrow 1$.

To avoid examples like this, we need a finitary analog of the Fubini property, which will serve as the pseudorandomness property we demand that our hypergraphs have.  As might be inferred from the example above, in finite hypergraphs we ask not for exact equality, but for approximate equality.


\begin{definition}
  Let $(G,\Gamma)$ be a $k$-uniform hypergraph with $G$ finite and let $V_0,V_1,W$ be disjoint sets and $E\subseteq {V_0\cup V_1\cup W\choose k}$.  We write $\mathcal{E}^\delta_{V_0,V_1,W,E}\subseteq\Gamma^W_{E,\emptyset}$ for the set of tuples $a_W\in\Gamma^W_{E,\emptyset}$ such that there is some partition $V=V_0\cup V_1$ such that
\[\int\left(\frac{|\Gamma^{V_1}_{E,x_{V_0}\cup a_W}|\cdot|\Gamma^{V_0}_{E,a_W}|}{|\Gamma^V_{E,a_W}|}-1\right)^2d\mu^{V_0}_{E,a_W}\geq\delta.\]

We say of $(G,\Gamma)$ that it \emph{$\delta$-consistently counts copies} of $(U,E)$ if whenever $V_0\cup V_1\cup W$ is a partition of $U$, $\mu^W_{E,\emptyset}(\mathcal{E}^\delta_{V_0,V_1,W,E})<\delta$.  We say of $(G,\Gamma)$ that it \emph{$\delta,d$-consistently counts copies} ($(\delta,d)$-ccc) if whenever $|U|\leq kd$ and $|E|\leq d$, $\Gamma$ $\delta$-consistently counts copies of $(U,E)$.
\end{definition}

Note that $\frac{|\Gamma^V_{E,a_W}|}{|\Gamma^{V_0}_{E,a_W}|}$ is the average number of ways that a copy of $V_0$ can be extended to a copy of $V_1$---that is, $\int\Gamma^{V_1}_{E,x_{V_0}\cup a_W}d\mu^{V_0}_{E,a_W}=\frac{|\Gamma^V_{E,a_W}|}{|\Gamma^{V_0}_{E,a_W}|}$.  Then consistently counting copies requires that most $x_{V_0}$ actually have close to an average number of extensions.  The further complication is that we allow an exceptional set of background parameters, the set $\mathcal{E}^\delta_{V_0,V_1,W,E}$, so long as this set is small.

In the graph case, this follows from the more familiar notion of \emph{bi-jumbled} graphs \cite{MR2319167} which shows up in proofs of sparse graph removal \cite{kohayakawa:MR1661982,ConlonFoxZhao}.  Recall that $\Gamma$ is $(p,\beta)$-bi-jumbled if for any sets of vertices $X,Y$,
\[\left|\,|(X\times Y)\cap \Gamma|-p|X|\cdot|Y|\,\right|\leq\beta\sqrt{|X|\cdot|Y|}.\]

\begin{lemma}
  For every graph $(U,E)$ there is a $d$ such that for each $\delta$ there is a $\gamma$ so that whenever $(G,\Gamma)$ is a $(p,\gamma p^dn)$-bi-jumbled graph with $n=|G|$ and $p=|\Gamma|/n^2$, $(G,\Gamma)$ $\delta$-consistently counts copies of $(U,E)$.
\end{lemma}

More generally, we wish to verify that $(\delta,d)$-ccc is a notion of pseudorandomness.  To do ths, we need to verify that when $\Gamma$ is chosen randomly and not too sparsely then, with high probability, $\Gamma$ $(\delta,d)$-consistently counts copies of $(U,E)$.  More precisely, for $p\in(0,1)$, fix the distribution $\mathbb{G}^k(n,p)$ on $k$-uniform hypergraphs with $n$ vertices is given by fixing a set $G$ of vertices with $|G|=n$ and assigning probability $p^e(1-p)^{{G\choose k}-e}$ to each hypergraph $\Gamma$ with $e$ vertices.

Equivalently, $\Gamma$ is chosen randomly according to $\mathbb{G}^k(n,p)$ if each edge belongs to $\Gamma$ independently with probability $p$.  With high probability, $\Gamma$ should have roughly $p{n\choose k}$ edges.  When we say $\Gamma$ is sparse, we mean we are considering the case where $p\rightarrow 0$ as $n\rightarrow\infty$.  By ``not too sparse'', we mean that $p$ is on the order of $n^{-1/r}$ for some positive natural number $r$.

The first step is to prove that, when $r$ is large enough, most random graphs have roughly the right number of small subgraphs.  For graphs this is quite standard \cite{MR809996}, and the proof readily generalizes to hypergraphs \cite{MR2592366}.
\begin{lemma}\label{thm:ccc_first_step}
  For any $k$-uniform hypergraph $(U,E)$, there is a sufficiently large $r$ so that for any $\delta,\epsilon>0$, for sufficiently large $n$, whenever $|G|=n$ and $\Gamma\subseteq{G\choose k}$ is chosen randomly according to $\mathbb{G}^k(n,n^{-1/r})$, with probability $\geq 1-\epsilon$,
\[\left||\Gamma^U_{E,\emptyset}|-n^{|U|}p^{|E|}\right|<\delta n^{|U|}p^{|E|}.\]
\end{lemma}
\begin{proof}
For each $x_U\in {G\choose U}$, let $I_{x_U}$ be the indicator variable which is equal to $1$ exactly when $x_U\in\Gamma^U_{E,\emptyset}$.  For any $x_U$, we have $\mathbb{E}(I_{x_U})=p^{|E|}$.
We have
\[|\Gamma^U_{E,\emptyset}|=\sum_{x_U\in{G\choose U}}I_{x_U},\]
so $\left|\mathbb{E}(|\Gamma^U_{E,\emptyset}|)-n^{|U|}p^{|E|}\right|<\delta n^{|U|}p^{|E|}/2$.

Then we have
\begin{align*}
    \mathbb{V}\mathrm{ar}(|\Gamma^U_{E,\emptyset}|)
&=\mathbb{V}\mathrm{ar}(\sum_{x_U}I_{x_U})\\
&=\sum_{x_U,y_U}\mathbb{C}\mathrm{ov}(I_{x_U},I_{y_U})\\
&=\sum_{x_U,y_U}(\mathbb{E}(I_{x_U}I_{y_U})-\mathbb{E}(I_{x_U})\mathbb{E}(I_{x_V}))\\
&=\sum_{x_U,y_U}(\mathbb{P}(I_{x_U}=1,I_{y_U}=1)-(\mathbb{P}(I_{x_U}=1))^2)\\
&=\sum_{x_U,y_U}(\mathbb{P}(I_{x_U}=1,I_{y_U}=1)-p^{2|E|}).
\end{align*}

$\mathbb{P}(I_{x_U}=1,I_{y_U}=1)$ is equal to $p^{2|E|-|E'|}$ where $(U',E')$ is the induced sub-hypergraph of $(U,E)$ isomorphic to the overlap between $x_U$ and $y_U$.  The terms where $|F'|=0$---that is, where $x_U$ and $y_U$ are disjoint (or at least have disjoint edges)---vanish.  For each induced sub-hypergraph $(U',E')$ of $(U,E)$, there are $\Theta(n^{2|U|-|U'|})$ pairs $x_U,y_U$ whose overlap is $(U',E')$, so for some $C_0,C_1$ depending only on $(U,E)$ (and independent of $r,n,\epsilon,\delta$)
\begin{align*}
\mathbb{V}\mathrm{ar}(|\Gamma^U_{E,\emptyset}|)
&\leq \sum_{(U',E')\subseteq(U,E),|E'|\neq 0}C_0 n^{2|U|-|U'|}(p^{2|E|-|E'|}-p^{2|E|})\\
&\leq \sum_{(U',E')\subseteq(U,E),|E'|\neq 0} C_1 n^{2|U|-|U'|}p^{2|E|-|E'|}.
\end{align*}
So picking $U'\subseteq U$ maximizing $n^{2|U|-|U'|}p^{2|E|-|E'|}$, we have
\[\mathbb{V}\mathrm{ar}(|\Gamma^U_{E,\emptyset}|)\leq C^2n^{2|U|-|U'|}p^{2|E|-|E'|}\]
for some $C$ independent of $r,n,\epsilon,\delta$.

In particular, by Chebyshev's inequality, for $\epsilon>0$, 
\begin{align*}
&\mathbb{P}(\left||\Gamma^U_{E,\emptyset}|-n^{|U|}p^{|E|}\right|\geq\delta n^{|U|}p^{|E|}/2
)\\
\leq &\mathbb{P}(\left||\Gamma^U_{E,\emptyset}|-\mathbb{E}(n^{|U|}p^{|E|})\right|\geq (\delta/2) n^{|U|}p^{|E|}
)\\
=&\mathbb{P}(\left||\Gamma^U_{E,\emptyset}|-\mathbb{E}(n^{|U|}p^{|E|})\right|\geq (\delta n^{|U'|/2}p^{|E'|/2}/4C)(Cn^{|U|-|U'|/2}p^{|E|-|E'|/2}))\\
\leq&\frac{4C^2}{\delta^2 n^{|U'|}p^{|E'|}}.
\end{align*}
When $r>\frac{|E'|}{|U'|}$, we have $n^{|U'|}p^{|E'|}>n^{|U'|-|E'|/r}\rightarrow\infty$, so this probability gets small as $n\rightarrow\infty$.  So by choosing $n$ large enough, we can ensure that $\frac{4C^2}{\delta^2n^{|U'|}p^{|E'|}}<\epsilon$, and so with probability $\geq 1-\epsilon$,
\[\left||\Gamma^U_{E,\emptyset}|-n^{|U|}p^{|E|}\right|<\delta n^{|U|}p^{|E|}\]
with high probability.
\end{proof}

\begin{lemma}\label{thm:ccc_intermediate_step}
  For any $k$-uniform hypergraph $(V\cup W,E)$ there is a sufficiently large $r$ so that for any $\delta,\epsilon>0$ and sufficiently large $n$, when $|G|=n$ and $\Gamma\subseteq{G\choose k}$ is chosen randomly according to $\mathbb{G}^k(n,n^{-1/r})$, with probability $\geq 1-\epsilon$,
 \[\int\left(|\Gamma^V_{E,x_W}|-\frac{|\Gamma^{V\cup W}_{E,\emptyset}|}{|\Gamma^W_{E,\emptyset}|}\right)^2d\mu^W_{E,\emptyset}<\delta n^{2|V|}p^{2(|E|-|E\upharpoonright[W]^2|)}.\]
\end{lemma}
\begin{proof}
  Let $E_W=E\upharpoonright [W]^2$ and $E_V=E\setminus E_W$.  We choose $r$ large enough so that we can apply the previous lemma to several hypergraphs based on $(V\cup W,E)$, to be determined in the course of the proof.  We can then choose $\delta'$ sufficiently small relative to $\delta$.  Then, by the previous lemma, when $n$ is large enough we have
\[\left||\Gamma^W_{E,\emptyset}|-n^{|W|}p^{|E_W|}\right|<\delta' n^{|W|}p^{|E_W|}.\]


Let $I^{x_W}(x_V)$ be the indicator function which is $1$ when $x_V\in\Gamma^V_{E,x_W}$, so $|\Gamma^V_{E,x_W}|=\sum_{x_V\in{G\choose V}}I^{x_W}(x_V)$.  Expanding the integral gives
\begin{align*}
  \int\left(|\Gamma^V_{E,x_W}|-\frac{|\Gamma^{V\cup W}_{E,\emptyset}|}{|\Gamma^W_{E,\emptyset}|}\right)^2d\mu^W_{E,\emptyset}
&=\int \sum_{x_V\in {G\choose V},y_V\in {G\choose V}}I^{x_W}(x_V)I^{x_W}(y_V)d\mu^W_{E,\emptyset} -\left(\frac{|\Gamma^{V\cup W}_{E,\emptyset}|}{|\Gamma^W_{E,\emptyset}|}\right)^2.
\end{align*}

Consider the graph $((V\times\{0,1\})\cup W,E+V)$ whose edges have the form
\[(J\cap W)\cup \{(v,\omega(v))\mid v\in J\cap V\}\]
for $J\in E$ and $\omega:J\cap V\rightarrow\{0,1\}$.  Then we have
\begin{align*}
  \int \sum_{x_V\in {G\choose V},y_V\in {G\choose V}}I^{x_W}(x_V)I^{x_W}(y_V)d\mu^W_{E,\emptyset} 
&=\int \sum_{x_V\in {G\choose V},y_V\in {G\choose V}, x_V\cap y_V=\emptyset}I^{x_W}(x_V)I^{x_W}(y_V)d\mu^W_{E,\emptyset}\\
&\ \ \ \ \ \ \ \ +\int \sum_{x_V\in {G\choose V},y_V\in {G\choose V}, x_V\cap y_V\neq\emptyset}I^{x_W}(x_V)I^{x_W}(y_V)d\mu^W_{E,\emptyset}\\
&=\frac{|\Gamma^{((V\times\{0,1\})\cup W,E+V)}_{E+V,\emptyset}|}{|\Gamma^W_{E,\emptyset}|}\\
&\ \ \ \ \ \ \ \ +\int \sum_{x_V\in {G\choose V},y_V\in {G\choose V}, x_V\cap y_V\neq\emptyset}I^{x_W}(x_V)I^{x_W}(y_V)d\mu^W_{E,\emptyset}.
\end{align*}
When $r$ is sufficiently large (depending only on $(V,E)$), we have
\[\int \sum_{x_V\in {G\choose V},y_V\in {G\choose V}, x_V\cap y_V\neq\emptyset}I^{x_W}(x_V)I^{x_W}(y_V)d\mu^W_{E,\emptyset}\leq Cn^{2|V|-1}<(\delta/2) n^{2|V|}p^{2|E_V|}\]
and $\left||\Gamma^{((V\times\{0,1\})\cup W,E+V)}_{E+V,\emptyset}|-n^{|W|+2|V|}p^{|E_W|+2|E_V|}\right|<\delta' n^{|W|+2|V|}p^{|E_W|+2|E_V|}$
once $n$ is sufficiently large.

Therefore
\[\int\left(|\Gamma^V_{E,x_W}|-\frac{|\Gamma^{V\cup W}_{E,\emptyset}|}{|\Gamma^W_{E,\emptyset}|}\right)^2d\mu^W_{E,\emptyset}<\delta n^{2|V|}p^{2|E_V|}.\]
\end{proof}

\begin{cor}\label{thm:ccc_inter}
  For any $k$-uniform hypergraph $(V\cup W,E)$ there is a sufficiently large $r$ so that for any $\delta,\epsilon>0$ and sufficiently large $n$, when $|G|=n$ and $\Gamma\subseteq{G\choose k}$ is chosen randomly according to $\mathbb{G}^k(n,n^{-1/r})$, with probability $\geq 1-\epsilon$, the set of $a_W$ such that
\[\left||\Gamma^V_{E,a_W}|-n^{|V|}p^{|E|-|E\upharpoonright [W]^2|}\right|\geq \delta n^{|V|}p^{|E|-|E\upharpoonright [W]^2|}\]
has size $<\epsilon n^{|W|}p^{|E\upharpoonright [W]^2|}$.
\end{cor}
\begin{proof}
By the previous lemma together with Chebyshev's inequality, we can choose $n$ large enough that the set of $a_W$ such that $\left||\Gamma^V_{E,a_W}|-\frac{|\Gamma^{V\cup W}_{E,\emptyset}|}{|\Gamma^W_{E,\emptyset}|}\right|\geq (\delta/2)n^{|V|}p^{|E\upharpoonright [W]^2|}$ has size $\leq\epsilon$ with probability $\geq 1-\epsilon/2$.  Additionally, using Lemma \ref{thm:ccc_first_step}, we can choose $n$ large enough that with probability $\geq 1-\epsilon/2$, $\left|\frac{|\Gamma^{V\cup W}_{E,\emptyset}|}{|\Gamma^W_{E,\emptyset}|}-n^{|V|}p^{|E|-|E\upharpoonright [W]^2|}\right|<(\delta/2)n^{|V|}p^{|E\upharpoonright [W]^2|}$.
\end{proof}

\begin{theorem}\label{thm:ccc}
For any $k$-uniform hypergraph $(U,E)$ there is a sufficiently large $r$ so that for any $\delta,\epsilon>0$ and sufficiently large $n$, when $|G|=n$ and $\Gamma\subseteq{G\choose k}$ is chosen randomly according to $\mathbb{G}^k(n,n^{-1/r})$, with probability $\geq 1-\epsilon$, $\Gamma$ $\delta$-consistently counts copies of $(U,E)$.
\end{theorem}
\begin{proof}
Consider any partition $U=V_0\cup V_1\cup W$.  Corollary \ref{thm:ccc_inter} allows us to show that $\Gamma^{V_0\cup V_1}_{E,a_W}$ and $\Gamma^{V_0\cup(V_1\times\{0,1\})}_{E+V_1,a_W}$ have close to the right size for most $a_W$.  We may repeat the proof of Lemma \ref{thm:ccc_intermediate_step} for each such $a_W$ to show that
\[\int\left(|\Gamma^{V_0\cup W}_{E,x_{V_0}\cup a_W}|-\frac{|\Gamma^{V_0\cup V_1\cup W}_{E,a_W}|}{|\Gamma^{V_0\cup W}_{E,a_W}|}\right)^2d\mu^W_{E,a_W}<\delta n^{2|V_1|}p^{2(|E|-|E\upharpoonright[V_0\cup W]^2|)}\]
\end{proof}

To our knowledge, $(\delta,d)$-ccc is not quite identical to any other notion of pseudorandomness in the literature \cite{MR1909084,MR2864650}.  When thinking in terms of model theory, as we are here, it is natural to consider only those sets which are definable, and as a result, all the sets we consider have a rate of growth from some small fixed list---in our case, we only need to worry about the behavior of sets $X$ where $|X|$ is $O(n^p)$ for finitely many choices of $p$.  The definition of notions like bi-jumbledness, on the other hand, ranges over all possible sets.  In practice, however, this is no difference; a given proof using a pseudorandomness assumption only uses the assumption for a fixed list of sets, which have a fixed list of rates of growth.  In particular, we expect that the proofs in \cite{kohayakawa:MR1661982,ConlonFoxZhao} go through unchanged if one replaces the assumption that the ambient graph is $(p,\beta)$-bi-jumbled with the assumption that it is $(\delta,d)$-ccc with an appropriate choice of parameters.

We observe that $\delta,d$-ccc implies an approximate version of the Fubini property.
\begin{lemma}\label{thm:approx_fubini}
Suppose $a_W\not\in\mathcal{E}^{\delta^2}_{V,W,E}$.  Then whenever $V=V_0\cup V_1$ is a non-trivial partition and $A\subseteq\Gamma^V_{E,a_W}$,
\[\left|\mu^V_{E,a_W}(A)-\int \mu^{V_1}_{E,x_{V_0}\cup a_W}(A(x_{V_0}))\, d\mu^{V_0}_{E,a_W}\right|<\delta.\]  
\end{lemma}
\begin{proof}
We have
  \begin{align*}
    \mu^V_{E,a_W}(A)
&=\frac{|A\cap\Gamma^V_{E,a_W}|}{|\Gamma^V_{E,a_W}|}\\
&=\sum_{x_{V_0}\in\Gamma^{V_0}_{E,a_W}}\frac{|A(x_{V_0})\cap\Gamma^{V_1}_{E,x_{V_0}\cup a_W}|}{|\Gamma^V_{E,a_W}|}\\
&=\int |A(x_{V_0})\cap\Gamma^{V_1}_{E,x_{V_0}\cup a_W}|\frac{|\Gamma^{V_0}_{E,a_W}|}{|\Gamma^V_{E,a_W}|}d\mu^{V_0}_{E,a_W}\\
&=\int \frac{|A(x_{V_0})\cap\Gamma^{V_1}_{E,x_{V_0}\cup a_W}|}{|\Gamma^{V_1}_{E,x_{V_0}\cup a_W}|}d\mu^{V_0}_{E,a_W}\\
&\ \ \ \ -\int \frac{|A(x_{V_0})\cap\Gamma^{V_1}_{E,x_{V_0}\cup a_W}|}{|\Gamma^{V_1}_{E,x_{V_0}\cup a_W}|}\left(\frac{|\Gamma^{V_0}_{E,a_W}|\cdot |\Gamma^{V_1}_{E,x_{V_0}\cup a_W}|}{|\Gamma^V_{E,a_W}|}-1\right)d\mu^{V_0}_{E,a_W}.\\
&=\int \mu^{V_1}_{E,x_{V_0}\cup a_W}(A(x_{V_0}))d\mu^{V_0}_{E,a_W}\\
&\ \ \ \ -\int \mu^{V_1}_{E,x_{V_0}\cup a_W}(A(x_{V_0}))\left(\frac{|\Gamma^{V_0}_{E,a_W}|\cdot |\Gamma^{V_1}_{E,x_{V_0}\cup a_W}|}{|\Gamma^V_{E,a_W}|}-1\right)d\mu^{V_0}_{E,a_W}.\\
  \end{align*}
So it suffices to observe that, since $\mu^{V_1}_{E,x_{V_0}\cup a_W}(A(x_{V_0}))\leq 1$,
\begin{align*}
  &\left|\int \mu^{V_1}_{E,x_{V_0}\cup a_W}(A(x_{V_0}))\left(\frac{|\Gamma^{V_0}_{E,a_W}|\cdot |\Gamma^{V_1}_{E,x_{V_0}\cup a_W}|}{|\Gamma^V_{E,a_W}|}-1\right)d\mu^{V_0}_{E,a_W}\right|\\
\leq&\sqrt{\int \mu^{V_1}_{E,x_{V_0}\cup a_W}(A(x_{V_0}))^2d\mu^{V_0}_{E,a_W}\int\left(\frac{|\Gamma^{V_0}_{E,a_W}|\cdot |\Gamma^{V_1}_{E,x_{V_0}\cup a_W}|}{|\Gamma^V_{E,a_W}|}-1\right)^2d\mu^{V_0}_{E,a_W}}\\
\leq&\sqrt{\delta^2}.
\end{align*}
since $a_W\not\in\mathcal{E}^{\delta^2}_{V,W,E}$.
\end{proof}

\section{Models}\label{model_theory}

In this section we deal with the passage from a sequence of finite hypergraphs (for our purposes, a sequence of hypothetical counterexamples to hypergraph removal) to a single infinitary hypergraph.  This requires the use of some model theory to produce an ultraproduct with suitable properties.  

We will refer to our models as $\mathfrak{M},\mathfrak{N}$, and to the corresponding universes of these models as $M,N$ respectively.  We will refer to formal variables in the language of first-order logic with the letter $w$, reserving the letters $x,y$ and so on for elements of models (for instance, when integrating over a model).  We will often refer to fixed elements of a model (used as constants or parameters) with the letters $a,b,c$.  In keeping with our tuple notation, we will often refer to finite sets of variables as $w_V,w_W$, etc..

 Recall that when $\varphi$ is a formula with free variables $w_V$, $\mathfrak{M}$ is a model of first-order logic, and $x_V\in M^V$, we write $\mathfrak{M}\vDash\varphi(x_V)$ to indicate that the formula holds when we interpret each free variable $w_v$ by the element $x_v$.  A set $B\subseteq M^V$ is \emph{definable} if $B=\{x_V\mid\mathfrak{M}\vDash\varphi(x_V)\}$ for some formula $\varphi$.  When the model $\mathfrak{M}$ is clear from context, we will often equate formulas with the sets they define---for instance, if $B$ is a definable set, we will also consider $B$ to be the formula defining this set, so by abuse of notation, $B=\{x_V\mid\mathfrak{M}\vDash B(x_V)\}$.  We say $B$ is \emph{definable from parameters} if $B=C(a_W)$ for some definable set $C$.

 Similarly, when $f$ is a simple function built from sets definable from parameters, so $f=\sum_{i\leq n}\alpha_i\chi_{C_i}$ where each $\alpha_i$ is rational and each $C_i$ is definable from parameters, we some view $f$ as being a ``rational linear combination'' of formulas, and refer to the union of the parameters defining all the sets $C_i$ as the parameters of $f$.

From here on, for any $V$, we understand $\mathcal{B}^0_V$ to be the collection of sets of $V$-tuples definable (with parameters) in the model $\mathfrak{M}$ (which will always be clear from context).

In our infinitary setting, we no longer have the underlying counting measures to refer to, so we will have to define formally the properties we want a family of measures to have.  We will use the meta-variable $\mu$ for a \emph{family of probability measures}---technically, a function from appropriate finite sets to probability measures, so when $\mu$ is a family of probability measures, $\mu^V_{E,x_W}$ is an actual probability measure for suitable values of $V$, $E$, $x_W$.

\begin{definition}
 Let $\mathfrak{M}$ be a model.  We say $\mu$ is a \emph{weakly canonical family of probability measures of degree $k$ and size $d$} if for any finite sets $V,W$ with $V\cap W=\emptyset$, any $k$-uniform hypergraph $E$ on $V\cup W$ with $|E|\leq d$, and any $x_W\in M^W$, a probability measure $\mu^V_{E,x_W}$ on $\mathcal{B}_V$ such that:
 \begin{enumerate}
 \item For $\mu^W_{E,\emptyset}$-almost-every $x_W$, $\mu^V_{E,x_W}$ is a definable Keisler probability measure,
\item If no edges in $E$ contain both $w$ and an element of $V$ then $\mu^V_{E,x_{W\cup\{w\}}}=\mu^V_{E,x_W}$,
\item If $\pi:V_0\cup W_0\rightarrow V_1\cup W_1$ is a bijection mapping $V_0$ to $V_1$ and $W_0$ to $W_1$ and $\pi(E_0)=E_1$ then $\mu^{V_0}_{E_0,x_{W_0}}=\mu^{V_1}_{E_1,x_{\pi(W_0)}}$.
 \end{enumerate}

We say $\mu$ is a \emph{canonical family of probability measures} if additionally
\begin{enumerate}
\item[(4)] When $V=V_0\cup V_1$, $V_0\cap V_1=\emptyset$, for $\mu^W_{E,\emptyset}$-almost every $x_W$ these measures satisfy the Fubini properties
\[\int\cdot\, d\mu^V_{E,x_W}=\iint\cdot\, d\mu^{V_1}_{E,x_{V_0}\cup x_W}d\mu^{V_0}_{E,x_W}.\]
\end{enumerate}
\end{definition}
Weak canonicity merely enforces a certain amount of uniformity on these measures---the second condition requires that $x_w$ only matters if there is an edge connecting $w$ to a vertex in $V$, and the third condition says that the measures depend only on the isomorphism class of the hypergraph $(V\cup W,E)$, not the particular choice of indices to represent it.  The Fubini condition is non-trivial, and it is ensuring this property that requires us to work only with sufficiently pseudorandom sparse hypergraphs.

The Fubini property obviously implies that we can exchange the order of integrals:
\[\iint\cdot\, d\mu^{V_1}_{E,x_{V_0}\cup x_W}d\mu^{V_0}_{E,x_W}=\iint\cdot\, d\mu^{V_0}_{E,x_{V_1}\cup x_W}d\mu^{V_1}_{E,x_W}.\]
In this form, the property is essentially the measure-invariance property which characterizes \emph{graphings}, which are the limits of extremely sparse graphs (usually bounded degree, and therefore much sparser than the graphs discussed in this paper); see \cite{elek:MR2359831}.

The natural language to begin with is a language with two $k$-ary relations, one for the ambient hypergraph $\Gamma$ and one for a sub-hypergraph $A$.  We wish to work in models which have two additional features: first, the model actually includes formulas defining all of the measures in the family $\mu$.  Second, for technical reasons, the model contains extra function symbols $\mathbf{max}$ which pick out values maximizing certain integrals.  (The construction of such languages has appeared a few times: see \cite{hrushovski,towsner09:gowers}, and a general theory of constructions of this kind is given in \cite{goldbring:_approx_logic_measure}.)
\begin{definition}
Let $\mathcal{L}$ be a language of first-order logic containing a $k$-ary relation symbol $\gamma$, and let $d$ be given.  $\mathcal{L}^{\gamma,d}$ is the smallest language containing $\mathcal{L}$ such that:
\begin{itemize}
\item Whenever $\varphi(w_V,w_W,w_P)$ is a formula with the displayed free variables, $W$ is a set disjoint from $V$, $E$ is a $k$-uniform hypergraph on $V\cup W$ with $|E|\leq d$, and $q\in[0,1]$ is rational, there are formulas
\[ m^V_{E,w_W}\leq q. \varphi\]
and
\[ m^V_{E,w_W}< q. \varphi\]
with free variables $w_W,w_P$, and
\item Whenever $E$ is a $k$-uniform hypergraph with $\leq d$ edges on a vertex set $V$, $V=V_0\cup V_1$ is a partition of $V$, $W$ and $P$ are finite sets with $V,W,P$ pairwise disjoint, $f$ is a rational linear combination of formulas with free variables $w_W,w_V$, and $\varphi(w_W,w_P,w_{V})$ is a formula with the displayed free variables, for each $p\in P$ there is a function symbol $\mathbf{max}^{E,V_0,f,\varphi}_p(w_W,w_{V_0})$.
\end{itemize}
\end{definition}
Note that the formulas $m^V_{E,w_W}\leq q.\varphi$ and $m^V_{E,w_W}< q. \varphi$ bind the variables $w_V$.  We will ``abbreviate'' these formulas as $m^V_{E,w_W}(\varphi)\leq q$ and $m^V_{E,w_W}(\varphi)<q$ respectively.  We will abbreviate $\neg m^V_{E,w_W}(\varphi)\leq q$ by $m^V_{E,w_W}(\varphi)> q$ and $\neg m^V_{E,w_W}(\varphi)< q$ by $m^V_{E,w_W}(\varphi)\geq q$.  We view $\{\mathbf{max}^{E,V_0,f,\varphi}_p(w_W,w_P,w_{V_0})\}_{p\in P}$ as a tuple $\mathbf{max}^{E,V_0,f,\varphi}_{P}(w_W,w_{V_0})$ of function symbols.

\begin{definition}
If $\mathfrak{M}$ is a finite model of $\mathcal{L}$ and $\Gamma=\gamma^{\mathfrak{M}}$ is the interpretation of $\gamma$ in this model, we expand $\mathfrak{M}$ to a model $\mathfrak{M}^{\Gamma,d}$ of $\mathcal{L}^{\gamma,d}$ by interpreting, for any $a_W\in M^W$,
\[ \mathfrak{M}^{\Gamma,d}\vDash m^V_{E,a_W}(B)\leq q\]
to hold iff
\[\mu^V_{E,a_W}(B)\leq q\]
whenever $B$ is definable from parameters, and similarly for $m^V_{E,a_W}(B)< q$.

Suppose we have interpreted the formula $\varphi$ and all the formulas defining the simple function $f$.  Let $B$ be the set defined by $\varphi$.  For each $a_W\in M^W, x_{V_0}\in M^{V_0}$, we choose $(\mathbf{max}_{P}^{E,V_0,f,\varphi}(a_W,x_{V_0}))^{\mathfrak{M}^{\Gamma,d}}$ to be some tuple $b_P$ maximizing $\left|\int f \chi_{B(a_W,x_{V_0},b_P)}\,d\mu^{V_1}_{E,x_{V_0}}\right|$.
\end{definition}
Note that we consistently use $m$ to refer to the formula of first-order logic describing a measure, and $\mu$ to the actual measure corresponding to $m$.  Also, note that in the interpretation of $\mathbf{max}_{P}^{E,V_0,f,\varphi}(a_W,x_{V_0})$, $B$ depends on $a_W,x_{V}$, and $b_P$, while $f$ depends on only $a_W$ and $x_V$.

Let $\mathcal{L}$ be the language consisting of two $k$-ary relation symbols, $\gamma$ and $\alpha$.

\begin{theorem}\label{ultraproduct}
Let $\epsilon>0$.  Suppose that for each $n$, $\Gamma_n$ is a $\delta_n,d$-ccc $k$-uniform hypergraph where $\delta_n\rightarrow 0$, and let $A_n\subseteq \Gamma_n$ be given with $|A_n|\geq\epsilon|\Gamma_n|$.  Then each $\mathfrak{M}_n=(\Gamma_n,A_n)$ is a model of $\mathcal{L}$.  Let $\mathcal{U}$ be an ultrafilter on $\mathbb{N}$ and let $\mathfrak{M}$ be the ultraproduct of the models $\mathfrak{M}^{\Gamma_n,d}_n$.  Then $\mathfrak{M}$ is a model of $\mathcal{L}^{\gamma,d}$ such that:
  \begin{enumerate}
  \item $\mathfrak{M}\vDash\sigma$ iff for $\mathcal{U}$-almost-every $n$, $\mathfrak{M}_n^{\Gamma_n,d}\vDash\sigma$.
  \item There is a canonical family of probability measures of degree $k$ and size $d$, $\mu^V_{E,x_W}$ on the $\sigma$-algebra generated by the definable subsets of $M^V$ such that whenever $B$ is definable from parameters,
\[\mu^V_{E,a_W}(B)=\inf\{q\in\mathbb{Q}^{>0}\mid \mathfrak{M}\vDash m^V_{E,a_W}(B)<q\}.\]
\item $\mu^{[1,k]}_{\{[1,k]\}}(A)\geq\epsilon$.
\item Whenever $E$ is a $k$-uniform hypergraph with $\leq d$ edges on a vertex set $V\cup W$, $V=V_0\cup V_1$ is a partition of $V$, $W$ and $P$ are finite sets with $V,W,P$ pairwise disjoint, $f$ is a rational linear combination of formulas with free variables $w_W,w_V$, and $\varphi(w_W,w_P,w_{V})$ is a formula with the displayed free variables, for almost every $a_W\in M^W, b_P\in M^P, x_{V_0}\in M^{V_0}$,
\[\left|\int f\chi_{B(a_W,x_{V_0},\mathbf{max}_P^{E,V_0,f,\varphi}(a_W,x_{V_0}))}\,d\mu^{V_1}_{E,x_{V_0}}\right|\geq\left|\int f\chi_{B(a_W,x_{V_0},b_P)}\,d\mu^{V_1}_{E,x_{V_0}}\right|.\]
\end{enumerate}
\end{theorem}
\begin{proof}
\begin{enumerate}[leftmargin=0pt,itemindent=1.6em,itemsep=1em]
\item  The first part is the standard \L o\'s Theorem for ultraproducts.
\item That the measures $\mu^V_{E,a_W}$ defined as in the statement extend to genuine probability measures on $\mathcal{B}_V$ is the standard Loeb measure construction.  The measures $\mu^V_{E,a_W}$ in the finite models are uniformly definable Keisler measures, and so the measures $\mu^V_{E,a_W}$ are definable Keisler measures as well (see \cite{goldbring:_approx_logic_measure} for details).  This satisfies the first requirement of weak canonicity.  The second and third requirements in the definition of weak canonicity are implied by formulas saying that certain measures are equal---for instance, the second requirement is implied by formulas of the form
\[\forall x_W\forall x_w\forall y(m^V_{E,x_{W\cup \{w\}}}\leq q.\phi(z,x_W,y)\leftrightarrow m^V_{E,x_W}\leq q.\phi(z,x_W,y)).\]
These formulas are all satisfied in all the finite models, and so by the first part, are also satisfied in $\mathfrak{M}$.  It follows that the family $\mu$ is weakly canonical of degree $k$ and size $d$.

Note that the formulas satisfied by $m^V_{E,a_W}$ in $\mathfrak{M}$ and the actual measure $\mu^V_{E,a_W}$ \emph{almost} line up: when $B$ is definable from parameters, if $\mu^V_{E,a_W}(B)<q$ then $\mathfrak{M}\vDash m^V_{E,a_W}(B)<q$, but if $\mathfrak{M}\vDash m^V_{E,a_W}(B)<q$ then we can only be sure that $\mu^V_{E,a_W}(B)\leq q$.

To see that the measures $\mu^V_{E,a_W}$ are actually canonical, it suffices to show that for each $B\in\mathcal{B}^0_V$ and $\mu^W_{E,\emptyset}$-almost every $x_W\in M^W$,
\[\mu^V_{E,x_W}(B(x_W))=\int\mu^{V_1}_{E,x_{V_0}\cup x_W}(B(x_W))\, d\mu^{V_0}_{E,x_W}.\]
Suppose not; then for some set $B$ definable from parameters, there is a set of $x_W$ of positive measure such that this equality fails.  It follows that for some rational $\delta>0$ there is a set $X_0$ of $x_W$ of positive measure such that 
\[\left|\mu^V_{E,x_W}(B(x_W))-\int\mu^{V_1}_{E,x_{V_0}\cup x_W}(B(x_W))\, d\mu^{V_0}_{E,x_W}\right|>\delta.\]
We need to approximate the integral in this definition closely enough by a formula to let us define a set of points where this violation occurs.  Consider the function $f_{x_W}(x_{V_0})=\mu^{V_1}_{E,x_{V_0}\cup x_W}(B(x_W))$.  We have $0\leq f_{x_W}(x_{V_0})\leq 1$. 

Integrals are not directly definable in our language, and there are many ways a function could have a given integral---for instance, by having a small number of points where the value is large, or a larger number of points where the value is smaller.  However we will show that there must be a set of positive measure where the functions $f_{x_W}$ not only all have nearly the same integral, but all these integrals can be finitely approximated using the same level sets.  This will allow us to write down a formula defining a set of points of positive measure, and with the property that every point satisfying this formula belongs to $X_0$.

  We may partition the interval $[0,1]$ into finitely many intervals $I_i=[\delta_i,\delta_{i+1})$ of size $<\delta/8$ and with rational endpoints.  Let us set $\Pi_i(x_W)=\{x_{V_0}\mid f_{x_W}(x_{V_0})\in I_i\}$ and $\pi_i(x_W)=\mu^{V_0}_{E,x_W}(\Pi_i(x_W))$, so when $x_W\in X_0$, $\sum_i\delta_i\pi_i(x_w)\leq\int f_{x_W}\,d\mu^{V_0}_{E,x_W}< \sum_i\delta_i\pi_i(x_W)+\delta/8$.

We choose $X_1\subseteq X_0$ of positive measure and, for each $i$, an interval $J_i=(\eta_i,\eta'_i)$ with rational end points such that $\pi_i(x_W)\in J_i$ for each $x_W\in X_1$ and 
\[\sum_i\delta_{i+1}\eta'_i<\sum_i\delta_i\eta_i+\delta/4.\]

Choose a rational $\sigma>0$ very small, and let
\begin{align*}
\Pi'_i(x_W)=\{x_{V_0}\mid \mathfrak{M}\vDash &\ m^{V_1}_{E,x_{V_0}\cup x_W}(B(x_W))<\delta_{i+1}\\
\wedge&\ m^{V_1}_{E,x_{V_0}\cup x_W}(B(x_W))>\delta_i-\sigma\}.
\end{align*}
Then $\Pi_i(x_W)\subseteq \Pi'_i(x_W)$ and $\Pi'_i(x_W)$ is definable.  Let $\pi'_i(x_W)=\mu^{V_0}_{E,x_W}(\Pi'_i(x_W))$.  By choosing $\sigma$ small enough, we may find a set $X_2\subseteq X_1$ of positive measure so that for $x_W\in X_2$, each $\pi'_i(x_W)\in J_i$ as well.

Now we may consider the set $\Theta$ of $x_W$ such that
\[\forall i \left(\mathfrak{M}\vDash m^{V_0}_{E,x_W}(\Pi'_i(x_W)) <\eta'_i\wedge  m^{V_0}_{E,x_W}(\Pi'_i(x_W))>\eta_i\right).\]
Note that $\Theta$ is definable from parameters and $X_2\subseteq\Theta$.

Consider any $x_W\in\Theta$, not necessarily in $X_2$.  Since each $\mu^{V_0}_{E,x_W}(\Pi'_i(x_W))\leq\eta'_i$,
\[\int\mu^{V_1}_{E,x_{V_0}\cup x_W}(B(x_W))\, d\mu^{V_0}_{E,x_W}\leq \sum_i\delta_{i+1}\eta'_i<\sum_i\delta_i\eta_i+\delta/4.\]
  On the other hand, since each $\mu^{V_0}_{E,x_W}(\Pi'_i(x_W))\geq\eta_i$,
\[\int\mu^{V_1}_{E,x_{V_0}\cup x_W}(B(x_W))\, d\mu^{V_0}_{E,x_W}\geq\sum_i(\delta_i-\sigma)\eta_i>\sum_i\delta_i\eta_i-\delta/4\]
 (since we chose $\sigma$ small enough).

So when $x_W\in\Theta$, we have
\[\sum_i(\delta_i-\sigma)\eta_i-\delta/4< \int\mu^{V_1}_{E,x_{V_0}\cup x_W}(B)\, d\mu^{V_0}_{E,x_W}<\sum_i\delta_i\eta_i+\delta/4.\]
Therefore when $x_W\in X_2\subseteq X_0\cap\Theta$, we must have either $\mu^{V}_{E,x_{W}}(B(x_W))< \sum_i\delta_i\eta_i-\delta/2$ or $\mu^{V}_{E,x_{W}}(B(x_W))> \sum_i\delta_i\eta_i+\delta/2$, and therefore
\begin{align*}
\mathfrak{M}\vDash &\left(m^V_{E,x_W}(B(x_W))<\sum_i\delta_i\eta_i-\delta/2\right)\\
\vee& \left(m^V_{E,x_W}(B(x_W))>\sum_i\delta_i\eta_i+\delta/2\right).
\end{align*}
Let $\psi$ be the conjunction of this formula with the formula defining $\Theta$.  Then we have $\mathfrak{M}\vDash\psi(x_W)$ whenever $x_W\in X_2$, and therefore $\mathfrak{M}\vDash  m^W_E(\psi)>\zeta$ for some $\zeta>0$.  We also have that whenever $\mathfrak{M}\vDash\psi(x_W)$, it is actually true that $\left|\mu^V_{E,x_W}(B(x_W))-\int\mu^{V_1}_{E,x_{V_0}\cup x_W}(B(x_W))\,d\mu^{V_0}_{E,x_W}\right|>\delta$.

Since the formula $ m^W_E(\psi)>\zeta$ holds in the ultraproduct, it also holds in infinitely many finite models.  But any finite model where this holds fails to satisfy the conclusion of Lemma \ref{thm:approx_fubini}, and therefore fails to be $\zeta,d$-ccc.  This contradicts the assumption that the finite models are $\delta_n,d$-ccc with $\delta_n\rightarrow 0$.
\item The third requirement follows immediately the \L o\'s Theorem: the formula $ m^{[1,k]}_{\{[1,k]\}}(A)\geq\epsilon$ holds in every finite model, and therefore in $\mathfrak{M}$ as well, and therefore $\mu^{[1,k]}_{\{[1,k]\}}(A)\geq\epsilon$.
\item Fortunately, the integral in this statement does not cause as much difficulty, since we do not need to deal with it uniformly in parameters.  Let $f=\sum\alpha_i\chi_{C_i}$.  Whenever $\left|\int f\chi_{B(a_W,x_{V_0},b_P)}\,d\mu^{V_1}_{E,x_{V_0}}\right|>\epsilon$ for some $\epsilon$, there is a formula holding of the parameters $a_W,x_{V_0},b_P$ which is a conjunction of components of the form
\[m^{V_1}_{E,x_{V_0}}(C_i(a_W,x_{V_0})\wedge B(a_W,b_P,x_{V_0}))<q\]
or negations of such components, and which implies that the integral is $\geq\epsilon$.  But then this formula holds in $\mathcal{U}$-almost every finite model, which means that we must have $\left|\int f\chi_{B(a_W,x_{V_0},\mathbf{max}_P^{E,V_0,f,\varphi}(a_W,x_{V_0}))}\,d\mu^{V_1}_{E,x_{V_0}}\right|\geq\epsilon$ in $\mathcal{U}$-almost every finite model (where $a_P$, etc., refer to the corresponding parameters in those finite models).  But then this formula also holds in $\mathfrak{M}$, so $\left|\int f\chi_{B(a_W,x_{V_0},\mathbf{max}_P^{E,V_0,f,\varphi}(a_W,x_{V_0}))}\,d\mu^{V_1}_{E,x_{V_0}}\right|\geq\epsilon$ in $\mathfrak{M}$.  Since this holds for every $\epsilon <\left|\int f\chi_{B(a_W,x_{V_0},b_P)}\,d\mu^{V_1}_{E,x_{V_0}}\right|$, it follows that
\[\left|\int f\chi_{B(a_W,x_{V_0},\mathbf{max}_P^{E,V_0,f,\varphi}(a_W,x_{V_0}))}\,d\mu^{V_1}_{E,x_{V_0}}\right|\geq\left|\int f\chi_{B(a_W,x_{V_0},b_P)}\,d\mu^{V_1}_{E,x_{V_0}}\right|.\]
\end{enumerate}
\end{proof}

\section{Uniformity Seminorms}\label{seminorms}

We give an outline of the remainder of our proof.  We will work in the setting established in the previous section---an infinite hypergraph together with a family of measures satisfying Fubini's theorem---and by our work in Section \ref{hypergraph_removal}, it will suffice to show that these measures have regularity.  In order to do this we will introduce a family of seminorms, the Gowers uniformity seminorms \cite{gowers:MR1844079}, which will correspond with the $\sigma$-algebras we introduced in Section \ref{sigma_algebras}.  (The connection between the Gowers seminorms and hypergraph regularity has been well-studied \cite{MR2373376,MR2426176,MR2195580}.  Infinitary versions were introduced by Host and Kra \cite{host05}, and have also been well-studied \cite{MR2920990,towsner09:gowers,2009arXiv0911.1157S,2009arXiv0903.0897S}.)

We want these seminorms to have the property that the seminorm ${||\cdot||_{U^{V,\mathcal{I}^\bot}_\infty(\mu^V_{E,a_P})}}$ corresponds to the $\sigma$-algebra $\mathcal{B}_{V,<V}$ in the sense that
\[||f||_{U^{V,\mathcal{I}^\bot}_\infty(\mu^V_{E,a_P})}=0\Leftrightarrow ||\mathbb{E}(f\mid\mathcal{B}_{V,\mathcal{I}})||_{L^2(\mu^V_{E,a_P})}=0.\]
We develop the seminorms in three stages: we define the \emph{principal seminorms}, which correspond to the principal $\sigma$-algebras $\mathcal{B}_{V,<V}$; the \emph{simple nonprincipal seminorms}, which correspond to the $\sigma$-algebras $\mathcal{B}_{V,<J}$ for $J\subsetneq V$; and the \emph{compound nonprincipal seminorms}, which correspond to the remaining $\sigma$-algebras.  The left to right implication is fairly easy to show (Theorem \ref{thm:easy_principal} for the principal seminorms and Theorem \ref{forwards} for nonprincipal seminorms).

We will call $U^V_\infty(\mu^V_{E,a_P})$ \emph{characteristic} when the right to left implication holds.  We will show that when the right seminorms are characteristic, the measure $\mu^V_{E,a_P}$ has regularity (Theorem \ref{char_implies_regular}), and therefore it will suffice to show that the seminorms are characteristic.
 
We define a special class of measures generalizing the dense setting---\emph{product measures}---and the structure of our argument is as follows (with all theorems assuming we have a canonical family of measures of sufficient size and degree):
\begin{enumerate}
\item Principal seminorms over a product measure are characteristic.  This argument is essentially standard; we give it in Theorem \ref{converse_product}.
\item \label{s2}Simple nonprincipal seminorms over a product measure are characteristic (Lemma \ref{converse_nonprincipal_base}).
\item \label{s3}All seminorms over a product measure are characteristic (Theorem \ref{converse_nonprincipal}).  This step proceeds inductively, using the inductive hypothesis with Theorem \ref{char_implies_regular} to show that the measure has $J$-regularity.
\item Principal seminorms over arbitrary measures are characteristic (Theorem \ref{converse_random}).
\item We now repeat (\ref{s2}) and (\ref{s3}) over an arbitrary measure, showing that all seminorms are characteristic.
\end{enumerate}

\subsection{Seminorms for Principal Algebras}\label{sec:principal_seminorms}
Fix disjoint sets $V,P$ and a $k$-uniform hypergraph $E\subseteq{{V\cup P}\choose k}$; let $m=|E\cap {P\choose k}|$ and let $\mu$ be a canonical family of measures of degree $k$ and size $\sum_{I\in E}2^{|I\cap V|}$.  Let $a_P$ be such that the measure $\mu^V_{E,a_P}$, and the measures we generate from it below, satisfy the appropriate Fubini properties.  (We will only work with a finite family of measures, so the set of such $a_P$ has $\mu^P_E$-measure $1$.)  To avoid repeating the background parameters $a_P$ over and over, we will write $\mu^V_E$ as an abbreviation for $\mu^V_{E,a_P}$ and $\mu^V_{E,x_W}$ as an abbreviation for $\mu^V_{E,x_W\cup a_P}$.

We wish to introduce the Gowers uniformity seminorms.  The basic idea is illustrated by the first non-trivial case: if $f(x_v,x_w)\in\mu^{\{v,w\}}_\emptyset$ then we have
\[||f||_{U^{\{v,w\}}_E}^4=\int f(x_v,x_{w})f(x_v,x_{w'})f(x_{v'},x_w)f(x_{v'},x_{w'})d\mu^{\{v,v',w,w'\}}_\emptyset.\]
We need to generalize this to the case where $f(x_v,x_w)\in\mu^{\{v,w\}}_{\{(v,w)\}}$; the correct choice is
\[||f||_{U^{\{v,w\}}_E}^4=\int f(x_v,x_{w})f(x_v,x_{w'})f(x_{v'},x_w)f(x_{v'},x_{w'})d\mu^{\{v,v',w,w'\}}_{\{(v,w),(v,w'),(v',w),(v',w')\}}.\]

We first need to define the general operation mapping a measure like $\mu^{\{v,w\}}_{\{(v,w)\}}$ to one like $\mu^{\{v,v',w,w'\}}_{\{(v,w),(v,w'),(v',w),(v',w')\}}$.
\begin{definition} 
For each $I\subseteq V$, we define $\mu^{V+I}_{E}=\mu^{(V\setminus I)\cup(I\times\{0,1\})}_{E^{V+I}}$ where $E^{V+I}$ is given as follows: for each $J\in E$ and each $\omega:J\cap I\rightarrow\{0,1\}$, there is an edge $J^\omega=(J\setminus I)\cup\{(i,\omega(i))\mid i\in J\cap I\}$.
\end{definition}
The graph $((V\setminus I)\cup(I\times\{0,1\}),E^{V+I})$ is the result of replacing the vertices $I$ with two identical copies of $I$.  In our example above, $\mu^{\{v,v',w,w'\}}_{\{(v,w),(v,w'),(v',w),(v',w')\}}=\mu^{\{v,w\}+\{v,w\}}_E$ (where, for greater generality, we have renamed $v$ to $(v,0)$, $v'$ to $(v,1)$, and similarly for $w,w'$).

Note that $\mu^{V+\emptyset}_{E}=\mu^V_{E}$.  For $i\in I$, $b\in \{0,1\}$, we write $x^b_i$ in place of $x_{(i,b)}$; for instance, we write
\[\int f( x_{V\setminus I}, x_I^0, x_I^1)\,d\mu^{V+I}_{E}\]
where the variables being integrated over are exactly the ones displayed.  If $\omega:I\rightarrow\{0,1\}$, we write $x^\omega_I$ for the tuple $x^{\omega}_I(i)=x^{\omega(i)}_i$.

Note that we chose the size of our measure to be $\sum_{I\in E}2^{|I\cap V|}$ because this is precisely the size needed to ensure Fubini properties for $\mu^{V+V}_E$.

\begin{definition} 
Let $f:M^V\rightarrow\mathbb{R}$ be an $L^\infty(\mu^V_E)$, $\mathcal{B}_V$-measurable function with $|V|=n$.  Define $\|\cdot\|_{U^V_\infty(\mu^{V}_{E})}$ by:
\[\|f\|_{U^V_\infty(\mu^V_{E})}=\left(\int \prod_{\omega\in\{0,1\}^V}f( x^\omega_V)\,d\mu^{V+V}_{E}\right)^{2^{-n}}.\]
\end{definition}

Whenever we refer to the norm $\|f\|_{U^V_\infty(\mu^V_E)}$, we assume that $f$ is $L^\infty(\mu^V_E)$ and $\mathcal{B}_V$-measurable.

We have to check that the expression under the radical is non-negative.  We actually prove the following stronger lemma, which will be useful later.
\begin{lemma}
If $f$ is an $L^\infty(\mu^V_{E})$ function and $B$ is $\mathcal{B}_{V,I}$-measurable for some $I\subsetneq V$ then
\[0\leq \int \prod_{\omega\in\{0,1\}^V}f( x^\omega_V)\chi_{B}(x^\omega_V)\,d\mu^{V+V}_{E}\leq \int \prod_{\omega\in\{0,1\}^V}f( x^\omega_V)\,d\mu^{V+V}_{E}.\]
\label{gowers_wf}
\end{lemma}
\begin{proof} 
It suffices to show the claim in the case when $|I|=|V|-1$.  Since $f=f\chi_B+f\chi_{\overline{B}}$, we have
\[\int \prod_{\omega\in\{0,1\}^V}f( x^\omega_V)\,d\mu^{V+V}_{E}=\int \prod_{\omega\in\{0,1\}^V}\left[(f\chi_B)( x^\omega_V)+(f\chi_{\overline{B}})(x^\omega_V)\right]\,d\mu^{V+V}_{E}.\]
Expanding the product gives a sum of $2^{2^n}$ terms of the form
\[\int \prod_{\omega\in\{0,1\}^V}(f\chi_{S_\omega})( x^\omega_V)\,d\mu^{V+V}_{E}\]
where each $S_\omega$ is either $\chi_B$ or $\chi_{\overline{B}}$.  We will show that each of these terms is non-negative.  Since $\int \prod_{\omega\in\{0,1\}^V}f( x^\omega_V)\chi_{B}(x^\omega_V)\,d\mu^{V+V}_{E}$ is one of these terms, both inequalities follow.

Note that $\chi_{S_\omega}(x^\omega_V)$ depends only on $x^\omega_I$.  In particular, if there are any $\omega,\omega'\in\{0,1\}^V$ such that $\omega(i)=\omega'(i)$ for all $i\in I$ but $S_\omega\neq S_{\omega'}$, then for any $x^0_V\cup x^1_V$, $\chi_{S_\omega}(x^\omega_V)=\chi_{S_\omega}(x^\omega_I)=\chi_{S_\omega}(x^{\omega'}_I)\neq \chi_{S_{\omega'}}(x^{\omega'}_I)=\chi_{S_{\omega'}}(x^{\omega'}_V)$.  In particular, one of these two values must be $0$, so the whole product is $0$.

So we may restrict to the case where $S_\omega$ depends only on $\omega\upharpoonright I$.  Let $v$ be the unique element in $V\setminus I$ and let $E'=E\upharpoonright{I\choose k}$.  Then we have the decomposition
\[\int\cdot\, d\mu^{V+V}_E=\iint\cdot\, d\mu^{v+v}_{E,x^0_I\cup x^1_I}d\mu^{I+I}_{E'}=\iiint \cdot\, d\mu^v_{E,x^0_I\cup x^1_I}d\mu^v_{E,x^0_I\cup x^1_I}d\mu^{I+I}_{E'}.\]
The second equality holds because the graph $E^{V+V}$ used to defined the measure $\mu^{V+V}_E$ does not contain any edges containing both $(v,0)$ and $(v,1)$.  So we have
\begin{align*}
& \int \prod_{\omega\in\{0,1\}^V}(f\chi_{S_\omega})( x_V^\omega)\, d\mu^{V+V}_{E}\\
& = \iint \prod_{\omega\in\{0,1\}^V}(f\chi_{S_\omega})( x_V^\omega)\, d\mu^v_{E,x^0_I\cup x^1_I}d\mu^v_{E,x^0_I\cup x^1_I}d\mu^{I+I}_{E'}\\
& =\iint \prod_{\omega\in\{0,1\}^{I}}f\chi_{S_\omega}( x_{I}^\omega,x_v^0)\prod_{\omega\in\{0,1\}^{I}}f\chi_{S_\omega}( x_{I}^\omega,x_v^1)\,d\mu^v_{E,x^0_I\cup x^1_I}d\mu^v_{E,x^0_I\cup x^1_I}d\mu^{I+I}_{E'}\\
& =\int \left(\int \prod_{\omega\in\{0,1\}^{I}}f\chi_{S_\omega}( x_{I}^\omega,x_v)\,d\mu^v_{E,x^0_I\cup x^1_I}\right)^2d\mu^{I+I}_{E'}\\
 \end{align*}
Since the inside of the integral is always non-negative, this term is non-negative.
\end{proof}

In particular, since $\int \prod_{\omega\in\{0,1\}^V}f( x^\omega_V)\,d\mu^{V+V}_{E}\geq 0$, $\|f\|_{U^V_\infty(\mu^V_E)}$ is defined.

Next we want a Cauchy-Schwarz style inequality for these seminorms:
\begin{lemma}[Gowers-Cauchy-Schwarz]
  Suppose that for each $\omega\in\{0,1\}^V$, $f_\omega$ is an $L^\infty(\mu^V_E)$ function.  Then
\[\left|\int \prod_{\omega\in\{0,1\}^V}f_\omega(x^\omega_V)\,d\mu^{V+V}_E\right|\leq \prod_{\omega\in\{0,1\}^V}\|f_\omega\|_{U^V_\infty(\mu^V_E)}.\]
\end{lemma}
\begin{proof}
Fix some $v\in V$, and let $I=V\setminus\{v\}$.  Note that we have the decomposition
\[\int\cdot\, d\mu^{V+V}_E=\iint \cdot\, d\mu^{v+v}_{E,x^0_{I}\cup x^1_{I}}d\mu^{I+I}_E=\iiint \cdot\, d\mu^v_{E,x^0_{I}\cup x^1_{I}}d\mu^v_{E,x^0_{I}\cup x^1_{I}}d\mu^{I+I}_E.\]
As above, the second equality holds because the graph in $\mu^{V+V}_E$ does not contain any edges containing both $(v,0)$ and $(v,1)$.  For $\omega\in\{0,1\}^{I}$ and $b\in\{0,1\}$, let us write $\omega b$ for the element of $\{0,1\}^V$ given by $(\omega b)(i)=\omega(i)$ if $i\in I$ and $(\omega b)(i)=b$ if $i=v$.  Therefore, using Cauchy-Schwarz, we have:
  \begin{align*}
 &   \left|\int \prod_{\omega\in\{0,1\}^V}f_\omega(x^\omega_V)\,d\mu^{V+V}_E\right|^2\\
 &=   \left|\int\left(\int \prod_{\omega\in\{0,1\}^V}f_\omega(x^\omega_{I},x^{\omega(v)}_v)\,d\mu^{v+v}_{E,x^0_{I}\cup x^1_{I}}\right)d\mu^{I+I}_E\right|^2\\
 &=   \left|\int\left(\int \prod_{\omega\in\{0,1\}^{I}}f_{\omega0}(x^\omega_{I},x^0_v)\,d\mu^{v}_{E,x^0_{I}\cup x^1_{I}}\right)\left(\int \prod_{\omega\in\{0,1\}^{I}}f_{\omega1}(x^\omega_{I},x^1_v)\,d\mu^{v}_{E,x^0_{I}\cup x^1_{I}}\right)d\mu^{I+I}_E\right|^2\\
 &\leq \int \left(\int \prod_{\omega\in\{0,1\}^{I}}f_{\omega0}(x^\omega_{I},x_v)\,d\mu^{v}_{E,x^0_{I}\cup x^1_{I}}\right)^2d\mu^{I+I}_E\int \left(\int \prod_{\omega\in\{0,1\}^{I}}f_{\omega1}(x^\omega_{I},x_v)d\mu^{v}_{E,x^0_{I}\cup x^1_{I}}\right)^2\,d\mu^{I+I}_E\\
 &\leq \int\prod_{\omega\in\{0,1\}^{V}}f_{(\omega\upharpoonright I)0}(x^\omega_V)\,d\mu^{V+V}_E\int\prod_{\omega\in\{0,1\}^{V}}f_{(\omega\upharpoonright I)1}(x^\omega_V)\,d\mu^{V+V}_E\\
  \end{align*}

In particular, applying this repeatedly to each coordinate in $V$, we have
\begin{align*}
\left|\int \prod_{\omega\in\{0,1\}^V}f_\omega(x^\omega_V)\,d\mu^{V+V}_E\right|^{2^V}
&\leq \prod_{\omega\in\{0,1\}^V}\int \prod_{\omega'\in\{0,1\}^V}f_\omega(x^{\omega'}_V)\,d\mu^{V+V}_E\\
&=\prod_{\omega\in\{0,1\}^V} \|f_\omega\|^{2^V}_{U^V_\infty(\mu^V_E)}.
\end{align*}
\end{proof}

\begin{cor}
  $\left|\int f\,d\mu^V_E\right|\leq \|f\|_{U^V_\infty(\mu^V_E)}$.
\end{cor}
\begin{proof}
  In the previous lemma, take $f_{0^V}=f$ and $f_{\omega}=1$ for $\omega\neq 0^V$.
\end{proof}

\begin{lemma}
  $\|\cdot\|_{U^V_\infty(\mu^V_E)}$ is a seminorm.
\end{lemma}
\begin{proof}
It is easy to see from the definition that $\|cf\|_{U^V_\infty(\mu^V_E)}=|c|\cdot\|f\|_{U^V_\infty(\mu^V_E)}$.  To see subadditivity, observe that $\|f+g\|_{U^V_\infty(\mu^V_E)}^{2^{|V|}}$ expands to a sum of $2^{2^{|V|}}$ integrals, each of which, by the previous lemma, is bounded by $\|f\|_{U^V_\infty(\mu^V_E)}^m\|g\|_{U^V_\infty(\mu^V_E)}^{2^{|V|}-m}$ for a suitable $m$.  In particular, this bound is precisely $\left(\|f\|_{U^V_\infty(\mu^V_E)}+\|g\|_{U^V_\infty(\mu^V_E)}\right)^{2^{|V|}}$ as desired.
\end{proof}

The work above gives:
\begin{theorem}\label{thm:easy_principal}
  If $\|\mathbb{E}(f\mid\mathcal{B}_{V,<V})\|_{L^2(\mu^V_E)}>0$ then $\|f\|_{U^V_\infty(\mu^V_E)}>0$.
\end{theorem}
\begin{proof}
If $\|\mathbb{E}(f\mid\mathcal{B}_{V,<V})\|_{L^2(\mu^V_E)}>0$ then we may find, for each $I\subseteq V$ with $|I|=|V|-1$, $B_I\in\mathcal{B}_{V,I}$ such that
\[0<\left|\int f\prod_I \chi_{B_I}\,d\mu^V_E\right|\leq \|f\prod_I\chi_{B_I}\|_{U^V_\infty(\mu^V_E)}.\]
By repeatedly applying Lemma \ref{gowers_wf}, once to each $I$, we have 
\[0<\|f\prod_I\chi_{B_I}\|_{U^V_\infty(\mu^V_E)}\leq \|f\|_{U^V_\infty(\mu^V_E)}.\]
\end{proof}

We will obtain the converse, which will show that $\|\mathbb{E}(f\mid\mathcal{B}_{V,<V})\|_{L^2(\mu^V_E)}>0$ iff $\|f\|_{U^V_\infty(\mu^V_E)}>0$, and in particular will enable us to show that $\mu$ has $J$-regularity.


\begin{definition}
  We say $\mu^V_E$ is a \emph{product measure} if no element of $E$ contains more than one element of $V$.
\end{definition}
(Recall that $\mu^V_E$ abbreviates $\mu^V_{E,a_P}$, so there may still be edges in $E$ connecting elements of $V$ to elements of $P$.)  We call such measures product measures because they are extensions of the ordinary product measure $\prod_{v\in V}\mu^v_E$.

\begin{theorem}\label{converse_product}
If $\mu^V_E$ is a product measure, and $\|\mathbb{E}(f\mid\mathcal{B}_{V,<V})\|_{L^2(\mu^V_E)}=0$ then $\|f\|_{U^V_\infty(\mu^V_E)}=0$.
\end{theorem}
\begin{proof}
This is essentially identical to the argument we gave for regularity for ordinary measures.  Suppose $\|\mathbb{E}(f\mid\mathcal{B}_{V,<V})\|_{L^2(\mu^V)}=0$.  We have
\begin{align*} 
\|f\|_{U^V_\infty(\mu^V_{E})}
=&\int f( x_V^0)\prod_{\omega\in\{0,1\}^V,\omega\neq 0^V}f( x_V^\omega)\,d\mu^{V+V}_E\\
=&\iint f( x_V^0)\prod_{\omega\in\{0,1\}^n,\omega\neq 0^V}f( x_V^\omega)\,d\mu^V_Ed\mu^V_E\\
\end{align*}
This last equality holds because $\mu^V_E$ is a product measure, and so the inner copy of $\mu^V_E$ does not depend on the choice of $x_V^1$.

Observe that, for every particular value of $ x_V^1$, $\prod_{\omega\in\{0,1\}^V,\omega\neq 0^V}f( x_V^\omega)$ is $\mathcal{B}_{V,<V}$-measurable, and therefore
\[\int f( x_V^0)\prod_{\omega\in\{0,1\}^V,\omega\neq 0^V}f( x_V^\omega)\,d\mu^V_E=0.\]
\end{proof}

\subsection{Seminorms for Nonprincipal Algebras}\label{sec:nonprincipal_seminorms}
We will need a more general family of seminorms corresponding to arbitrary algebras of the form $\mathcal{B}_{V,\mathcal{I}}$.

\begin{definition}
For $J\subseteq V$, define
\[\|f\|_{U^{V,J}_\infty(\mu^V_E)}=\left(\int \prod_{\omega\in\{0,1\}^J}f(x_{V\setminus J},x_J^\omega)\,d\mu^{V+J}_E\right)^{2^{-|J|}}.\]
\end{definition}
Note that $\|f\|_{U^{V,V}_\infty(\mu^V_E)}=\|f\|_{U^V_\infty(\mu^V_E)}$.

We need to generalize to norms $U^{V,\mathcal{J}}$ where $\mathcal{J}$ is a set.  A natural choice would be to take the product of $U^{V,J}$ over all $J\in\mathcal{J}$, but this is not a seminorm.  Instead we need the following form:
\begin{definition}
  Let $\mathcal{J}\subseteq\mathcal{P}(V)$ be a set such that if $J,J'\in\mathcal{J}$ are distinct then $J\not\subseteq J'$.  Then we define
\[\|f\|_{U^{V,\mathcal{J}}_\infty(\mu^V_E)}=\inf\sum_{i\leq k}\left(\prod_{J\in\mathcal{J}}\|f_i\|^{2^{|J|}}_{U^{V,J}_\infty(\mu^V_E)}\right)^{\frac{1}{\sum_{J\in\mathcal{J}}2^{|J|}}}\]
where the infimum is taken over all sequences $f_0,\ldots,f_k$ such that $f=\sum_{i\leq k}f_i$.
\end{definition}
It is not immediately obvious that $U^{V,J}_\infty$ and $U^{V,\{J\}}_\infty$ calculate the same value, but this will follow once we show that $U^{V,J}_\infty$ is a seminorm.

\begin{lemma}
  If $f$ is an $L^\infty(\mu^V_E)$ function then
\[0\leq \int \prod_{\omega\in\{0,1\}^J}f(x_{V\setminus J},x_J^\omega)\,d\mu^{V+J}_E.\]
\end{lemma}
\begin{proof}
  Let $V'=V\setminus J$.  We have
\begin{align*}
\int \prod_{\omega\in\{0,1\}^J}f(x_{V'},x_J^\omega)\,d\mu^{V+J}_E
&=\iint \prod_{\omega\in\{0,1\}^J}f(x_{V'},x_J^\omega)\,d\mu^{J+J}_{E,x_{V'}}d\mu^{V'}_E\\
&=\int \|f(x_{V'},\cdot)\|^{2^{|J|}}_{U^J_\infty(\mu^J_{E,x_{V'}})}\,d\mu^{V'}_E\\
&\geq 0.
\end{align*}
\end{proof}

\begin{lemma}
 $|\int f\,d\mu^V_E|\leq \|f\|_{U^{V,\mathcal{J}}_\infty(\mu^V_E)}$
\end{lemma}
\begin{proof}
  First consider the case where $\mathcal{J}$ is a singleton $\{J\}$.  Again, let $V'=V\setminus J$.
  \begin{align*}
    \left|\int f\,d\mu^V_E\right|^{2^{|J|}}
&=\left|\iint f\,d\mu^J_{E,x_{V'}}d\mu^{V'}_E\right|^{2^{|J|}}\\
&\leq \int \left|\int f\,d\mu^J_{E,x_{V'}}\right|^{2^{|J|}}d\mu^{V'}_E\\
&\leq \int \|f\|^{2^{|J|}}_{U^J_\infty(\mu^J_{E,x_{V'}})}\,d\mu^{V'}_E\\
&=\|f\|_{U^{V,\mathcal{J}}_\infty(\mu^V_E)}^{2^{|J|}}.\\
  \end{align*}

For the general case, first observe that, setting $c=\sum_{J\in\mathcal{J}}2^{|J|}$,
\begin{align*}
  \left|\int f\,d\mu^V_E\right|^{c}
&=\prod_{J\in\mathcal{J}}\left|\int f\,d\mu^V_E\right|^{2^{|J|}}\\
&\leq \prod_{J\in\mathcal{J}}\|f\|^{2^{|J|}}_{U^{V,J}_\infty(\mu^V_E)}.\\
\end{align*}
So if $f=\sum_{i\leq k}f_i$ we have
\[\left|\int f\,d\mu^V_E\right|\leq\sum_{i\leq k}\left|\int f_i\,d\mu^V_E\right|\leq\sum_{i\leq k}\left(\prod_{J\in\mathcal{J}}\|f_i\|^{2^{|J|}}_{U^{V,J}_\infty(\mu^V_E)}\right)^{\frac{1}{c}}.\]
This holds for any $\sum_{i\leq k}f_i$, so $\left|\int f\,d\mu^V_E\right|\leq \|f\|_{U^{V,\mathcal{J}}_\infty(\mu^V_E)}$.
\end{proof}

\begin{lemma}
   $\|\cdot\|_{U^{V,\mathcal{J}}_\infty(\mu^V_E)}$ is a seminorm.
\end{lemma}
 \begin{proof}
   Once again positive homogeneity is obvious from the definition, so we need only check that the triangle inequality holds.

We first consider the case where $\mathcal{J}$ is a singleton:
   \begin{align*}
     \|f+g\|_{U^{V,J}_\infty(\mu^V_E)}^{2^{|J|}}
 &=\int \|f+g\|^{2^{|J|}}_{U^J_\infty(\mu^J_{E,x_{V'}})}\,d\mu^{V'}_E\\
 &\leq\int \left(\|f\|_{U^J_\infty(\mu^J_{E,x_{V'}})}+\|g\|_{U^J_\infty(\mu^J_{E,x_{V'}})}\right)^{2^{|J|}}\,d\mu^{V'}_E\\
 &=\int\sum_{i\leq {2^{|J|}}}{{2^{|J|}}\choose i}\|f\|_{U^J_\infty(\mu^J_{E,x_{V'}})}^i\|g\|_{U^J_\infty(\mu^J_{E,x_{V'}})}^{2^{|J|}-i}\,d\mu^{V'}_E\\
 &=\sum_{i\leq {2^{|J|}}}{{2^{|J|}}\choose i}\int\|f\|_{U^J_\infty(\mu^J_{E,x_{V'}})}^i\|g\|_{U^J_\infty(\mu^J_{E,x_{V'}})}^{2^{|J|}-i}\,d\mu^{V'}_E.\\
 \end{align*}
Applying H\"older's inequality with $p=2^{|J|}/i$ (and therefore $q=1/(1-1/p)=2^{|J|}/(2^{|J|}-i)$ gives an upper bound of
\begin{align*}
\|f+g\|_{U^{V,J}_\infty(\mu^V_E)}^{2^{|J|}}
&\leq\sum_{i\leq {2^{|J|}}}{{2^{|J|}}\choose i}
\left(\int\|f\|_{U^J_\infty(\mu^J_{E,x_{V'}})}^{2^{|J|}}\,d\mu^{V'}_E\right)^{i/2^{|J|}}
\left(\int\|g\|_{U^J_\infty(\mu^J_{E,x_{V'}})}^{2^{|J|}}\,d\mu^{V'}_E\right)^{(2^{|J|}-i)/2^{|J|}}\\
  &\leq\sum_{i\leq {2^{|J|}}}{{2^{|J|}}\choose i}\left(\|f\|_{U^{V,J}_\infty(\mu^V_E)}^{2^{|J|}}\right)^{i/2^{|J|}}\left(\|g\|_{U^{V,J}_\infty(\mu^V_E)}^{2^{|J|}}\right)^{(2^{|J|}-i)/2^{|J|}}\\
  &=\sum_{i\leq {2^{|J|}}}{{2^{|J|}}\choose i}\|f\|_{U^{V,J}_\infty(\mu^V_E)}^{i}\|g\|_{U^{V,J}_\infty(\mu^V_E)}^{2^{|J|}-i}\\
  &=\left(\|f\|_{U^{V,J}_\infty(\mu^V_E)}+\|g\|_{U^{V,J}_\infty(\mu^V_E)}\right)^{2^{|J|}}\\
\end{align*}



For $|\mathcal{J}|>1$, we may use the fact that if $f=\sum_{i\leq k}f_i$ and $g=\sum_{j\leq m}g_j$ then $f+g=\sum_{i\leq k}f_i+\sum_{j\leq m}g_j$.
 \end{proof}

The main thing that makes the uniformity seminorms useful to us is that they easily pass across different measures:
\begin{lemma}
Let $J\subseteq V$ and $V'=V\setminus J$.  If $\|f\|_{U^J_\infty(\mu^J_E)}=0$ then for $\mu^{V'}_E$-almost-every $x_{V'}$, $\|f\|_{U^J_\infty(\mu^J_{E,x_{V'}})}=0$.
\end{lemma}
\begin{proof}
  \begin{align*}
    0
&=\|f\|_{U^J_\infty(\mu^J_E)}^{2^{|J|}}\\
&=\int\prod_{\omega\in\{0,1\}^J}f(x_J^\omega)\,d\mu^J_E\\
&=\int\prod_{\omega\in\{0,1\}^J}f(x_J^\omega)\int 1\, d\mu^{V'}_{E,x_J^0\cup x_J^1}d\mu^{J+J}_E\\
&=\int\prod_{\omega\in\{0,1\}^J}f(x_J^\omega)\,d\mu^{J+J}_{E,x_V'}d\mu^{V'}_E\\
&=\int\|f\|_{U^J_\infty(\mu^J_{E,x_{V'}})}^{2^{|J|}}\,d\mu^{V'}_E.\\
  \end{align*}
Therefore for $\mu^{V'}_E$-almost-every $x_{V'}$, $\|f\|_{U^J_\infty(\mu^J_{E,x_{V'}})}=0$.
\end{proof}

In order to associate these more general seminorms with the correct algebras, we introduce the following definition:
\begin{definition}
  If $\mathcal{I}\subseteq\mathcal{P}(V)$ is non-empty, we define $\mathcal{I}^{\bot}$ to be the set of $J\subseteq V$ such that:
  \begin{enumerate}
  \item There is no $I\in\mathcal{I}$ with $J\subseteq I$,
  \item If $J'\subsetneq J$ then there is an $I\in\mathcal{I}$ with $J'\subseteq I$.
  \end{enumerate}

We also set $J^-=\{I\subseteq V\mid J\not\subseteq I\}$.
\end{definition}
$\cdot^\bot$ and $\cdot^-$ depend on the choice of the ambient set $V$.  We note some useful properties of these definitions:
\begin{enumerate}
\item If $\mathcal{I}=\{I\subseteq V\mid |I|=|V|-1\}$ then $\mathcal{I}^{\bot}=\{V\}$,
\item $(J^-)^{\bot}=\{J\}$,
\item $\mathcal{I}^{\bot}$ always has the property that if $J,J'\in\mathcal{I}^{\bot}$ are distinct then $J\not\subseteq J'$, and
\item If $\mathcal{J}$ has the property that when $J,J'\in\mathcal{J}$ are distinct then $J\not\subseteq J'$ and $\mathcal{I}$ is the collection of $I$ such that $I\subsetneq J$ for some $J\in\mathcal{J}$, then $\mathcal{J}=\mathcal{I}^{\bot}$.
\end{enumerate}
These last two properties show that $\|\cdot\|_{U^{V,\mathcal{J}}_\infty(\mu^V_E)}$ is defined exactly when $\mathcal{J}=\mathcal{I}^{\bot}$ for some $\mathcal{I}$.

We will eventually show that when $\mu^V_E$ is nice enough, $\mathcal{B}_{V,\mathcal{I}}$ and $\bigcap_{J\in\mathcal{I}^\bot}\mathcal{B}_{V,J^-}$ agree up to $\mu^V_E$ measure $0$.

\begin{lemma}
If there is no $J\in\mathcal{J}$ such that $J\subseteq I$ and $B$ is $\mathcal{B}_{V,I}$-measurable then 
\[\|f\chi_B\|_{U^{V,\mathcal{J}}_\infty(\mu^V_E)}\leq \|f\|_{U^{V,\mathcal{J}}_\infty(\mu^V_E)}.\]
\end{lemma}
\begin{proof}
It suffices to show this for $\mathcal{J}$ a singleton $\{J\}$.  Write $V'=V\setminus J$.  Observe that for any fixed $x_{V'}$, the function $\chi_B(x_{V'},\cdot)$ is a $\mathcal{B}_{J,J\cap I}$-measurable function, where $J\cap I$ must be a proper subset of $J$.  So we have:
  \begin{align*}
\|f\chi_B\|_{U^{V,J}_\infty(\mu^V_E)}^{2^{|J|}}
&=\int \|f\chi_B\|_{U^J_\infty(\mu^J_{E,x_{V'}})}^{2^{|J|}}\,d\mu^{V'}_E\\
&\leq\int \|f\|_{U^J_\infty(\mu^J_{E,x_{V'}})}^{2^{|J|}}\,d\mu^{V'}_E\\
&=\|f\|_{U^{V,J}_\infty(\mu^V_E)}^{2^{|J|}}.
  \end{align*}
\end{proof}

\begin{theorem}\label{forwards}
  If $\|\mathbb{E}(f\mid\mathcal{B}_{V,\mathcal{I}})\|_{L^2(\mu^V_E)}>0$ then $\|f\|_{U^{V,\mathcal{I}^{\bot}}_\infty(\mu^V_E)}>0$.
\end{theorem}
\begin{proof}
  If $\|\mathbb{E}(f\mid\mathcal{B}_{V,\mathcal{I}})\|_{L^2(\mu^V_E)}>0$ then we may find, for each $I\in\mathcal{I}$, a set $B_I\in\mathcal{B}_{V,I}$, such that 
\[0<\left|\int f\prod_I\chi_{B_I}\,d\mu^V_E\right|\leq \|f\prod_I\chi_{B_I}\|_{U^{V,\mathcal{I}^{\bot}}_\infty(\mu^V_E)}.\]
Observe that for each $I\in\mathcal{I}$ we may apply the previous lemma, so we have
\[0<\|f\prod_I\chi_{B_I}\|_{U^{V,\mathcal{I}^{\bot}}_\infty(\mu^V_E)}\leq \|f\|_{U^{V,\mathcal{I}^{\bot}}_\infty(\mu^V_E)}.\]
\end{proof}

\subsection{Characteristic Seminorms}\label{sec:nonprincipal_characteristic}

\begin{definition}
  Let $\mu$ be a canonical family of measures of degree $k$ and size $\sum_{I\in E}2^{|I\cap V|}$.  For some $\mathcal{I}\subseteq \mathcal{P}(V)$, we say $U^{V,\mathcal{I}^{\bot}}_\infty(\mu^V_E)$ is \emph{characteristic} if for each $f\in L^\infty(\mathcal{B}_V)$, $\|f\|_{U^{V,\mathcal{I}^{\bot}}_\infty(\mu^V_E)}>0$ implies $\|\mathbb{E}(f\mid\mathcal{B}_{V,\mathcal{I}})\|_{L^2(\mu^V_E)}>0$.
\end{definition}

\begin{theorem}\label{char_implies_regular}
Suppose that $J\subseteq V$ and that whenever $\mathcal{I}\subseteq\mathcal{P}(J)$ is such that for distinct $J',J''\in\mathcal{I}$, $J'\not\subseteq J''$, $U^{J,\mathcal{I}}_\infty(\mu^J_E)$ is characteristic.  Then $\mu^V_E$ has $J$-regularity.
\end{theorem}
\begin{proof}
  Let $J\subsetneq V$ and $\mathcal{I}\subseteq\mathcal{P}(V)$ be given, and let $g$ and $f_I$ be as in the definition of regularity.  Let $h=g-\mathbb{E}(g\mid\mathcal{B}_{V,\mathcal{I}\wedge J})=g-\mathbb{E}(g\mid\mathcal{B}_{J,\mathcal{I}\wedge J})$ (viewing $g$ as a function on $\mathcal{B}_J$).  Since $\|\mathbb{E}(h\mid\mathcal{B}_{J,\mathcal{I}\wedge J})\|_{L^2(\mu^J_E)}=0$, by assumption we have $\|h\|_{U^{J,\mathcal{I}\wedge J}_\infty(\mu^J_E)}=0$.  Then we also have $\|h\|_{U^{J,\mathcal{I}\wedge J}_\infty(\mu^J_{E,x_{V\setminus J}})}=0$ for $\mu^{V\setminus J}_E$-almost-every $x_{V\setminus J}$.  (The exact choice of \emph{which} set of measure $1$ this holds on depends on the choice of representative of $h$.)

Including $x_{V\setminus J}$ as part of the background parameters, Theorem \ref{forwards} implies that $\|\mathbb{E}(h\mid\mathcal{B}_{J,\mathcal{I}\wedge J})\|_{L^2(\mu^J_{E,x_{V\setminus J}})}=0$, and so 
 \begin{align*}
   \int h\prod_{I}f_I\, d\mu^V_E
&=   \int h\prod_{I}f_I\, d\mu^{J}_{E,x_{V\setminus J}}d\mu^{V\setminus J}_E=0
 \end{align*}
since for every fixed $x_{V\setminus J}$, $\prod_{I}f_I$ is $\mathcal{B}_{J,\mathcal{I}\wedge J}$-measurable.
\end{proof}

Our goal in the remainder of this subsection is to reduce the problem of showing that the uniformity norms are characteristic to showing that the principal uniformity norms are characteristic.  We only need this for the case of a product measure, but we include the general argument for completeness.

\begin{lemma}
  Let $\mathcal{I}$ be given and let $J\subseteq V$.  If $\mu^V_{E}$ has $J$-regularity, $g\in L^2(\mathcal{B}_{V,J})$, and $\mathbb{E}(g\mid\mathcal{B}_{V,\mathcal{I}\wedge J})=0$ then $\mathbb{E}(g\mid\mathcal{B}_{V,\mathcal{I}})=0$.
\end{lemma}
\begin{proof}
Let such a $g$ be given, and for each $I\in\mathcal{I}$, let $f_I$ be $\mathcal{B}_{V,I}$-measurable, so $\prod_If_I$ is $\mathcal{B}_{V,\mathcal{I}}$-measurable.  Since $g=g-\mathbb{E}(g\mid\mathcal{B}_{V,\mathcal{I}\wedge J})$ and $\mu^V_E$ has $J$-regularity, 
\[\int g\prod_If_I\,d\mu^V_E=0,\]
and since this holds for any choice of $f_I$, $\mathbb{E}(g\mid\mathcal{B}_{V,\mathcal{I}})=0$.
\end{proof}

\begin{lemma}
  Let $\mathcal{I}$ be given and let $J\subseteq V$.  If $\mu^V_{E}$ has $J$-regularity then $\mathcal{B}_{V,\mathcal{I}\wedge J}=\mathcal{B}_{V,\mathcal{I}}\cap\mathcal{B}_{V,J}$ up to $\mu^V_E$-measure $0$.
\end{lemma}
\begin{proof}
By definition, we have $\mathcal{B}_{V,\mathcal{I}\wedge J}\subseteq\mathcal{B}_{V,\mathcal{I}}\cap\mathcal{B}_{V,J}$.

For the converse, suppose $B\in\mathcal{B}_{V,\mathcal{I}}\cap\mathcal{B}_{V,J}$.  Let $g=\chi_B-\mathbb{E}(\chi_B\mid\mathcal{B}_{V,\mathcal{I}\wedge J})$; since both $\chi_B$ and $\mathbb{E}(\chi_B\mid\mathcal{B}_{V,\mathcal{I}\wedge J})$ are each both $\mathcal{B}_{V,J}$-measurable and $\mathcal{B}_{\mathcal{I}}$-measurable, $g$ is as well.  Since by definition $\mathbb{E}(g\mid\mathcal{B}_{\mathcal{I}\wedge J})=0$, by the previous lemma, $\mathbb{E}(g\mid\mathcal{B}_{V,\mathcal{I}})=0$.  Since $g$ is $\mathcal{B}_{V,\mathcal{I}}$-measurable, $g=\mathbb{E}(g\mid\mathcal{B}_{V,\mathcal{I}})=0$ (as $L^2$ functions, of course).  Therefore $\chi_B=\mathbb{E}(\chi_B\mid\mathcal{B}_{V,\mathcal{I}\wedge J})$ (again, as $L^2$ functions), and so $B$ is within measure $0$ of being $\mathcal{B}_{V,\mathcal{I}}\cap\mathcal{B}_{V,J}$-measurable.
\end{proof}

\begin{lemma}
  For any $\mathcal{I},\mathcal{J}\subseteq\mathcal{P}(V)$, if $\mu^V_E$ has $J$-regularity for every $J\in\mathcal{J}$ then $\mathcal{B}_{V,\mathcal{I}\wedge \mathcal{J}}$ is $\mathcal{B}_{V,\mathcal{I}}\cap\mathcal{B}_{V,\mathcal{J}}$ up to $\mu^V_E$-measure $0$.\label{algebra_wedge}
\end{lemma}
\begin{proof}
  The direction $\mathcal{B}_{V,\mathcal{I}\wedge \mathcal{J}}\subseteq\mathcal{B}_{V,\mathcal{I}}\cap\mathcal{B}_{V,\mathcal{J}}$ is immediate from the definition.

For the converse, we may assume $\mathcal{J}=\{J_1,\ldots,J_n\}$ where $i\neq j$ implies $J_i\not\subseteq J_j$, and we proceed by induction on $n$.  When $n=1$ this is just the previous lemma.  Suppose the claim holds for $\mathcal{J}$ and we wish to show it for $\mathcal{J}\cup\{J\}$.  Note that
\[\mathcal{B}_{\mathcal{I}\wedge(\mathcal{J}\cup\{J\})}=\mathcal{B}_{(\mathcal{I}\wedge\mathcal{J})\cup(\mathcal{I}\wedge J)}=\mathcal{B}_{\mathcal{I}\wedge\mathcal{J}}\uplus\mathcal{B}_{\mathcal{I}\wedge J}.\]

It suffices to show that whenever $f$ is $\mathcal{B}_{V,\mathcal{I}}$-measurable then $\mathbb{E}(f\mid\mathcal{B}_{V,\mathcal{J}\cup\{J\}})$ is $\mathcal{B}_{V,\mathcal{I}\wedge(\mathcal{J}\cup\{J\})}$-measurable.  For any $f$, we have
\begin{align*}
  \mathbb{E}(f\mid\mathcal{B}_{V,\mathcal{J}\cup\{J\}})
&=\mathbb{E}(f\mid\mathcal{B}_{V,\mathcal{J}}\uplus\mathcal{B}_J)\\
&=\mathbb{E}(f\mid\mathcal{B}_{V,\mathcal{J}})+\mathbb{E}(f\mid\mathcal{B}_J)-\mathbb{E}(f\mid\mathcal{B}_{V,\mathcal{J}}\cap\mathcal{B}_J).\\
\end{align*}
When $f$ is $\mathcal{B}_{V,\mathcal{I}}$-measurable, $\mathbb{E}(f\mid\mathcal{B}_{V,\mathcal{J}}) -\mathbb{E}(f\mid\mathcal{B}_{V,\mathcal{J}}\cap\mathcal{B}_J)$ is $\mathcal{B}_{V,\mathcal{I}}\cap\mathcal{B}_{V,\mathcal{J}}$-measurable, and therefore, by the inductive hypothesis, $\mathcal{B}_{V,\mathcal{I}\wedge\mathcal{J}}$-measurable.  By the previous lemma, $\mathbb{E}(f\mid\mathcal{B}_J)$ is $\mathcal{B}_{V,\mathcal{I}\wedge J}$-measurable.  In particular, this means $\mathbb{E}(f\mid\mathcal{B}_{V,\mathcal{J}\cup\{J\}})$ is $\mathcal{B}_{\mathcal{I}\wedge\mathcal{J}}\uplus\mathcal{B}_{\mathcal{I}\wedge J}=\mathcal{B}_{\mathcal{I}\wedge(\mathcal{J}\cup\{J\})}$-measurable.
\end{proof}

\begin{lemma}\label{distribute_product}
  If $\mu^V_E$ has $J$-regularity for every $J\in\mathcal{I}^{\bot}$, $\mathcal{B}_{V,\mathcal{I}}$ is $\bigcap_{J\in\mathcal{I}^{\bot}}\mathcal{B}_{V,J^-}$ up to $\mu^V_E$-measure $0$.
\end{lemma}
\begin{proof}
  We have $\bigcap_{J\in\mathcal{I}^{\bot}}\mathcal{B}_{V,J^-}$ is $\mathcal{B}_{V,\bigwedge_{J\in \mathcal{I}^{\bot}}J^-}$ up to $\mu^V_E$-measure $0$ (it is easy to see that $\wedge$ is associative and commutative, so this follows by repeated application of Lemma \ref{algebra_wedge}).  We need only check that $\bigwedge_{J\in \mathcal{I}^{\bot}}J^-=\mathcal{I}$.

If $I\in\mathcal{I}$ (or even $I\subseteq I'\in\mathcal{I}$) then for every $J\in\mathcal{I}^{\bot}$, we have $J\not\subseteq I$, and therefore $I\in J^-$, and therefore $I\in\bigwedge_{J\in \mathcal{I}^{\bot}}J^-$.  Conversely, if there is no $I'\in\mathcal{I}$ such that $I\subseteq I'$ then there is a $J\subseteq I$ such that $J\in\mathcal{I}^{\bot}$, and therefore no $J'\in J^-$ such that $I\subseteq J'$, and therefore $I\not\in\bigwedge_{J\in \mathcal{I}^{\bot}}J^-$.
\end{proof}

In the following lemma we have to directly appeal to the definability structure of our $\sigma$-algebras.  This is for a good reason: the statement would not be true if we replaced our $\sigma$-algebras with, say, simple product algebras.

\begin{lemma}\label{converse_nonprincipal_base}
  Suppose that for $J'\subseteq J$, $U^{J'}_\infty(\mu^{J'}_{E})$ is characteristic and that for $\mu^{V\setminus J'}_E$-almost-every $x_{V'}$, $U^{J'}_\infty(\mu^{J'}_{E,x_{V'}})$ is characteristic.  Then $U^{V,J}_\infty(\mu^V_E)$ is characteristic.
\end{lemma}
\begin{proof}
Suppose $\|f\|_{U^{V,J}_\infty(\mu^V_E)}>0$, so, setting $V'=V\setminus J$, also
\[0<\|f\|_{U^{V,J}_\infty(\mu^V_E)}^{2^{|J|}}=\int \|f\|_{U^J_\infty(\mu^J_{E,x_{V'}})}^{2^{|J|}}\,d\mu^{V'}_E.\]

There must be a set $S_0\subseteq M^{V'}$ of positive measure such that, for $x_{V'}\in S_0$, $\|f\|_{U^J_\infty(\mu^J_{E,x_{V'}})}^{2^{|J|}}>0$.  Since almost every $U^J_\infty(\mu^J_{E,x_{V'}})$ is characteristic, for almost every $x_{V'}\in S_0$, we have $\|\mathbb{E}(f\mid\mathcal{B}_{J,<J})\|_{L^2(\mu^J_{E,x_{V'}})}>0$.  This means that for almost every $x_{V'}\in S_0$, we may choose a set $B(x_{V'},\vec a)\in\mathcal{B}^0_{J,<J}$ such that $|\int f\chi_{B(x_{V'},\vec a)}\, d\mu^J_{E,x_{V'}}|>0$.  Since $\mathcal{B}^0_{J,<J}$ was chosen to be the collection of definable sets, and there are only countably many formulas, by passing to a smaller set of positive measure we may choose a single formula $B$, independent of $x_{V'}$, so that for each $x_{V'}\in S_0$ there are parameters $a_Q(x_{V'})$ so that $|\int f\chi_{B(x_{V'},a_Q(x_{V'}))}\, d\mu^J_{E,x_{V'}}|>0$.  (Recall our notation---$B\in\mathcal{B}^0_{J\cup Q}$, and then for each $x_{V'}$ we specialize to the slice $a_Q(x_{V'})$.  There are uncountably many possible choices for $a_Q(x_{V'})$, so we cannot assume $a_Q$ is independent of $x_{V'}$.) 

We may choose an $\epsilon>0$, an approximation of $f$ by a simple function $f'$, and a set $S_1\subseteq S_0$ of positive measure so that for $x_{V'}\in S_1$, $|\int f'\chi_{B(x_{V'},a_Q(x_{V'}))}\, d\mu^J_{E,x_{V'}}|\geq\epsilon$.  Since $f'$ is simple, $f'$ is itself defined using finitely many formulas, which in turn have finitely many parameters $a_W$.

Recall the distinguished function symbols $\mathbf{max}^{E,J,f',B}_Q$; these symbols choose values $a_Q(x_{V'})$ maximizing the value of $|\int f'\chi_{B(x_{V'}, a_Q(x_{V'}))}\,d\mu^J_{E,x_{V'}}|$.  So, replacing $B(x_{V'}, a_Q(x_{V'}))$ with $\hat B(x_{V'},a_W)=B(x_{V'},\mathbf{max}^{E,J,f',B}_Q(x_{V'},a_W))$,
\[\left|\int f'\chi_{\hat B(x_{V'},a_W)}\,d\mu^J_{E,x_{V'}}\right|\geq\left| \int f'\chi_{B(x_{V'},a_Q(x_{V'}))}\,d\mu^J_{E,x_{V'}}\right|.\]
In particular, for each $x_{V'}\in S_1$, $\left|\int f'\chi_{\hat B(x_{V'},a_W)}\,d\mu^J_{E,x_{V'}}\right|\geq\epsilon$.  Note that $\hat B(x_{V'},a_W)\in\mathcal{B}^0_{J,<J}$ (viewing $\hat B(x_{V'},a_W)$ as a set of $J$-tuples) and therefore $\hat B(a_W)\in\mathcal{B}^0_{J,J^-}$ (viewing $\hat B(a_W)$ as a set of $V$-tuples).

We may partition $S_1=S_1^+\cup S_1^-$ where $x_{V'}\in S_1^+$ exactly when $\int f'\chi_{\hat B(x_{V'},a_W)}\,d\mu^J_{E,x_{V'}}\geq\epsilon$.  Clearly at least one of $S_1^+$ and $S_1^-$ has measure $\geq\mu^{V'}_E(S_1)/2$; without loss of generality, we assume $S_1^+$ does.  Since $f'$ is simple, we have $f'=\sum_{i\leq n}\alpha_i \chi_{C_i}$.  We may write a large union $D$ of sets consisting of those $x_{V'}$ such that 
\begin{align*}
&\left(m^J_{E,x_{V'}}(C_1(x_{V'},a_W)\cap \hat B(x_{V'},a_W))<\beta_1\wedge
m^J_{E,x_{V'}}(C_1(x_{V'},a_W)\cap \hat B(x_{V'},a_W))>\beta'_1\right)\\
\wedge&\cdots\\
\wedge&\left(m^J_{E,x_{V'}}(C_n(x_{V'},a_W)\cap \hat B(x_{V'},a_W))<\beta_n\wedge
m^J_{E,x_{V'}}(C_n(x_{V'},a_W)\cap \hat B(x_{V'},a_W))>\beta'_n\right)\\
\end{align*}
so that $\mu^{V'}_E(D\cap S_1^+)\geq(1-\delta)\mu^{V'}_E(S_1^+)$ and every element of $D$ satisfies
\[\int f'\chi_{\hat B(x_{V'},a_W)}\, d\mu^J_{E,x_{V'}}>\epsilon/2.\]
The formula defining this set has only free variables $x_{V'}$, so $D$ is {$\mathcal{B}_{V,V'}$\nobreakdash-measurable}.  Then
\[\int f'\chi_{\hat B(x_{V'},a_W)}\chi_D\, d\mu^V_E d\mu=\iint f'\chi_{\hat B(x_{V'},a_W)}\chi_D d\mu^J_{E,x_{V'}}\,d\mu^{V'}_E>\epsilon(1-\delta)\mu^{V'}_E(S_1)/2.\]
Since we chose $f'$ to be an arbitrarily close approximation of $f$, we may assume that $\|f-f'\|_{L^2(\mu^V_E)}<\epsilon(1-\delta)\mu^{V'}_E(S_1)/4$, and so we have
\[\int f\chi_{\hat B(x_{V'},a_W)}\chi_D\, d\mu^V_E d\mu>\epsilon(1-\delta)\mu^{V'}_E(S_1)/4>0.\]
Since $\chi_{\hat B(x_{V'},a_W)}\chi_D$ is $\mathcal{B}_{V,J^-}$-measurable, we are finished.
\end{proof}

\begin{theorem}\label{converse_nonprincipal}
Suppose that for every $J\in\mathcal{I}^\bot$ and every $J'\subseteq J$, $U^{J'}_\infty(\mu^{J'}_{E})$ is characteristic and that for $\mu^{V\setminus J'}_E$-almost-every $x_{V'}$, $U^{J'}_\infty(\mu^{J'}_{E,x_{V'}})$ is characteristic.  Then $U^{V,\mathcal{I}^\bot}_\infty(\mu^V_E)$ is characteristic.
\end{theorem}
\begin{proof}
We proceed by main induction on $|V|$.  In particular, if $V\in\mathcal{I}^{\bot}$ then the claim is given by the assumption, so we may assume that every element $J\in\mathcal{I}^{\bot}$ has $|J|<|V|$, and so by the inductive hypothesis, $U^{J,\mathcal{J}^{\bot}}_\infty(\mu^J_E)$ is characteristic for any $\mathcal{J}^{\bot}$ with $\mathcal{J}\subseteq\mathcal{P}(J)$.  Therefore by Lemma \ref{char_implies_regular}, $\mu^V_E$ has $J$-regularity for each $J\in\mathcal{I}^\bot$.

Suppose $\|\mathbb{E}(f\mid\mathcal{B}_{V,\mathcal{I}})\|_{L^2(\mu^V_E)}=0$.  By Lemma \ref{distribute_product}, $\mathcal{B}_{V,\mathcal{I}}=\bigcap_{J\in\mathcal{I}^{\bot}}\mathcal{B}_{V,J^-}$, and so $\|\mathbb{E}(f\mid\bigcap_{J\in\mathcal{I}^{\bot}}\mathcal{B}_{V,J^-})\|_{L^2(\mu^V_E)}=0$.  Let $\mathcal{I}^{\bot}=\{J_1,\ldots,J_r\}$.  Then we may define a sequence $f_0=f$, $f_{i+1}=\mathbb{E}(f_i\mid\mathcal{B}_{V,J_{i+1}^-})$, 
\[f=f_r+(f_{r-1}-f_r)+(f_{r-2}-f_{r-1})+\cdots+ (f_0-f_1).\]
Since $f_r=\mathbb{E}(f\mid\bigcap_{J\in\mathcal{I}^{\bot}}\mathcal{B}_{V,J^-})$, we have $f_r=0$.  We therefore have
\[\|f\|_{U^{V,\mathcal{I}^{\bot}}_\infty(\mu^V_E)}\leq\sum_{i< r}\left(\prod_{J\in\mathcal{I}^{\bot}}\|f_i-f_{i+1}\|^{2^{|J|}}_{U^{V,J}_\infty(\mu^V_E)}\right)^{\frac{1}{\sum_{J\in\mathcal{J}}2^{|J|}}}.\]
For each $i<r$, $f_i-f_{i+1}=f_i-\mathbb{E}(f_i\mid\mathcal{B}_{V,J_{i+1}^-})$.  In particular, $\|\mathbb{E}(f_i-f_{i+1}\mid\mathcal{B}_{V,J_{i+1}^-})\|_{L^2(\mu^V_E)}=0$, and therefore by the previous lemma, $\|f_i-f_{i+1}\|_{U^{V,J_{i+1}}_\infty(\mu^V_E)}=0$.  But this means the whole sum is $0$, and therefore $\|f\|_{U^{V,\mathcal{J}}_\infty(\mu^V_E)}=0$.
\end{proof}

\section{Principal Seminorms are Characteristic}\label{random_measures}

We now turn to our final argument, showing that the principal norms are always characteristic.  The construction is notationally complicated, so we first illustrate the idea for the simplest interesting case: $V=\{v,w\}$ and $E=\{(v,w)\}$.  (And, for simplicity of notation, no background parameters.)  Suppose $f$ is an $L^\infty$ function with $||f||_{U^V_\infty(\mu^V_E)}>0$.  This is equivalent to having
\[\int f(x_v^0,x_w^0)f(x_v^0,x_w^1)f(x_v^1,x_w^0)f(x_v^1,x_w^1)d\mu^{V+V}_E>0.\]

We can define a function
\[G(x_w^0,x_w^1)=\int f(x_w^0,x_v^1)f(x_w^1,x_v^1)d\mu^v_{E,x^0_w\cup x^1_w}.\]
If we take all three coordinates ($x_w^0,x_w^1,x_v^1$) into account here, we are looking at the measure $\mu^{W^+}_{E'}$ where $W^+=\{(w,0),(w,1),(v,1)\}$ and $E'=\{((w,0),(v,1)),((w,1),(v,1))\}$.  If we take $W\subseteq W^+$ to be $\{(w,0),(w,1)\}$, we have that $\mu^W_{E',x_v^1}$ and $\mu^W_{\emptyset}$ are product measures, and therefore $G(x_w^0,x_w^1)$ is measurable with respect to $\mathcal{B}_{W,1}$.  This means that we may approximate $G(x_w^0,x_w^1)$ arbitrarily well (in the $L^2(\mu^W_\emptyset)$ norm) by sums of the form
\[\sum g_{0,n}(x_w^0)g_{1,n}(x_w^1).\]
We may further assume that the $g_{b,n}$ are $L^\infty(\mu^W_\emptyset)$ functions since $G$ is.  By assumption, we have
\begin{align*}
  0
&<\int f(x_v^0,x_w^0)f(x_v^0,x_w^1)f(x_v^1,x_w^0)f(x_v^1,x_w^1)d\mu^{V+V}_E\\
&=\int f(x_v^0,x_w^0)f(x_v^0,x_w^1)G(x_w^0,x_w^1)d\mu^{V+\{w\}}_E,
\end{align*}
and so also, taking a good enough approximation,
\[0<\int f(x_v^0,x_w^0)f(x_v^0,x_w^1)\sum g_{0,n}(x_w^0)g_{1,n}(x_w^1)d\mu^{V+\{w\}}_E.\]
In particular, for some $n$, we have
\[0<\int f(x_v^0,x_w^0)f(x_v,x_w^1)g_{0,n}(x_w^0)g_{1,n}(x_w^1)d\mu^{V+\{w\}}_E.\]
Dropping the $0$ superscript and setting $h(x_v)=\int f(x_v,x_w^1)g_{1,n}(x_w^1)d\mu^w_{E,x_v}$, we have
\[0<\int f(x_v,x_w)g_{0,n}(x_w)h(x_v)d\mu^{V}_E.\]
Since $g_{0,n}(x_w)h(x_v)$ is $\mathcal{B}_{V,<V}$-measurable by definition, it follows that $||\mathbb{E}(f\mid\mathcal{B}_{V,<V})||_{L^2(\mu^V_E)}>0$.

\begin{theorem}\label{converse_random}
Suppose $\mu$ is a canonical family of measures of degree $k$ and size $\sum_{e\in E}2^{2|e|}$.  Then $U^V_\infty(\mu^V_{E})$ is characteristic.
\end{theorem}
\begin{proof}
We show a more general result:
\begin{quote}
Let $I\subseteq V$ be given and let $V'=V\setminus I$.  Let $\{f_\omega\}_{\omega\in \{0,1\}^\omega}$ be $L^\infty(\mu^V_E)$ functions such that
\[F(x_{V'},x^0_I,x^1_I)=\prod_{\omega\in\{0,1\}^I}f_{\omega}(x_{V'},x^\omega_I)\]
be such that
\[0<\int F(x_{V'},x^0_I,x^1_I)\, d\mu^{V+I}_{E}.\]
Then $\|\mathbb{E}(f_{0^I}\mid\mathcal{B}_{V,<V})\|_{L^2(\mu^V_{E})}>0$.
\end{quote}
The main result is then the case where $I=V$ and $f_{\omega}=f$ for all $\omega$.


We proceed by induction on $|I|$.  When $I=\emptyset$, this is trivial---$F=f_{0^\emptyset}$, and $0<\int f\,d\mu^V_E$ implies that $f$ has non-trivial projection onto the trivial $\sigma$-algebra, so certainly also onto $\mathcal{B}_{V,<V}$.   So assume $|I|>0$.

Fix some $v\in I$, and let $I'=I\setminus\{v\}$.  For each $\omega\in\{0,1\}^{I'}$ and each $b\in\{0,1\}$ we will write $\omega b$ for the corresponding elements of $\{0,1\}^I$.  We define a function
\[G(x_{V'},x^0_{I'},x^1_{I'})=\int\prod_{\omega\in\{0,1\}^{I'}}f_{\omega 1}(x_{V'},x^\omega_{I'},x^1_v)\,d\mu^v_{E,x_{V'}\cup x^0_{I'}\cup x^1_{I'}}.\]

Let $W=V'\cup(I'\times\{0,1\})$; recall that there is an $E'$ such that $\mu^{V\setminus\{v\}+I'}_E=\mu^W_{E'}$.  Each edge in $E'$ corresponds to an edge in $E$, and each $e\in E$ leads to at most $2^{|e\cap V|}$ edges in $E'$, so $\sum_{e\in E'}2^{|e\cap W|}\leq \sum_{e\in E}2^{2|e|}$.  (There is likely some room here for optimizing the exact size of the canonical family needed.)

Let $\mathcal{J}\subseteq \mathcal{P}(W)$ be the collection of subsets of the form
\[V'\cup\{(i,\omega(i))\mid i\in I'\}\]
for some $\omega\in\{0,1\}^{I'}$.  That is, $\mathcal{J}$ consists of those sets which contain $V'$ together with exactly one copy of each coordinate from $I'$.  The elements of $\mathcal{J}^{\bot}$ are pairs $J=\{(i,0),(i,1)\}$ for some $i\in I'$.  No edge of $E'$ contains both elements of a pair $\{(i,0),(i,1)\}$, so $\mu^J_{E'}$ and $\mu^J_{E',x_{W\setminus J'}}$ are product measures, and in particular, $U^J_\infty(\mu^J_{E'})$ and $U^J_\infty(\mu^J_{E',x_{W\setminus J}})$ are characteristic by Theorem \ref{converse_product}.

We claim that $G$ is $\mathcal{B}_{W,\mathcal{J}}$-measurable (with respect to the measure $\mu^W_{E'}$).  Suppose $H$ is a function with $\|\mathbb{E}(H\mid\mathcal{B}_{W,\mathcal{J}})\|_{L^2(\mu^W_{E'})}=0$.  By Theorem \ref{converse_nonprincipal}, $U^{W,\mathcal{J}^{\bot}}_\infty(\mu^W_{E'})$ is characteristic, so $\|H\|_{U^{W,\mathcal{J}^{\bot}}_\infty(\mu^W_{E'})}=0$, and therefore for $\mu^v_E$-almost-every $x^1_v$, $\|H\|_{U^{W,\mathcal{J}^{\bot}}_\infty(\mu^W_{E', x^1_v})}=0$, and so $\|\mathbb{E}(H\mid\mathcal{B}_{W,\mathcal{J}})\|_{L^2(\mu^W_{E',x^1_v})}=0$.  Then
\begin{align*}
&\ \ \ \;\int H(x_{V'},x^0_{I'},x^1_{I'})\cdot G(x_{V'},x^0_{I'},x^1_{I'})\, d\mu^W_{E'}\\
&=\iint H \prod_{\omega\in\{0,1\}^{I'}}f_{\omega 1}(x_{V'},x^\omega_{I'},x^1_v)\, d\mu^v_{E, x_{V'}\cup x^0_{I'}\cup x^1_{I'}} d\mu^W_{E'}\\
&=\iint H \prod_{\omega\in\{0,1\}^{I'}}f_{\omega 1}(x_{V'},x^\omega_{I'},x^1_v)\, d\mu^W_{E', x^1_v}d\mu^v_{E}\\
&=0.
\end{align*}
Since this holds for any $H$ with $\|\mathbb{E}(H\mid\mathcal{B}_{W,\mathcal{J}})\|_{L^2(\mu^W_{E'})}=0$, it follows that $G$ is $\mathcal{B}_{W,\mathcal{J}}$-measurable.  This means that we may write
\[ G(x_{V'},x^0_{I'},x^1_{I'})=\lim_{N\rightarrow\infty}\sum_{n\leq N}\prod_{\omega\in\{0,1\}^{I'}}g_{\omega,n,N}(x_{V'},x^\omega_{I'})\]
up to the $L^2(\mu^W_{E'})$-norm.  We may assume the $g_{\omega,n,N}$ are $L^\infty(\mu^W_{E'})$ functions.

Then we have some $\epsilon$ such that
\begin{align*}
  0<\epsilon
&<\int\prod_{\omega\in\{0,1\}^I}f_{\omega}(x_{V'},x^\omega_I)\, d\mu^{V+I}_{E}\\
&=\int\prod_{\omega\in\{0,1\}^{I'}}f_{\omega0}(x_{V'},x^\omega_{I'},x^0_v)\prod_{\omega\in\{0,1\}^{I'}}f_{\omega1}(x_{V'},x^\omega_{I'},x^1_v)\, d\mu^{V+I}_{E}\\
&=\int\prod_{\omega\in\{0,1\}^{I'}}f_{\omega0}(x_{V'},x^\omega_{I'},x^0_v) G(x_{V'},x^0_{I'},x^1_{I'})\, d\mu^{V+I'}_{E}.\\
\end{align*}
Choosing $N$ large enough, we may make
\begin{align*}
&\|G(x_{V'},x^0_{I'},x^1_{I'})-\sum_{n\leq N}\prod_{\omega\in\{0,1\}^{I'}}g_{\omega,n,N}(x_{V'},x^\omega_{I'})\|_{L^2(\mu^W_{E'})}\\
\end{align*}
arbitrarily small relative to $\epsilon$ and the $L^\infty(\mu^{V+I'}_E)$-norms of the $f_{\omega0}$; therefore
\begin{align*}
0<\epsilon/2
&<\int\prod_{\omega\in\{0,1\}^{I'}}f_{\omega0}(x_{V'},x^\omega_I,x^0_v) \sum_{n\leq N}\prod_{\omega\in\{0,1\}^{I'}}g_{\omega,n,N}(x_{V'},x^\omega_{I'})\,d\mu^{V+I'}_{E}\\
&= \sum_{n,N}\int\prod_{\omega\in\{0,1\}^{I'}}f_{\omega0}(x_{V'},x^\omega_I,x^0_v)g_{\omega,n,N}(x_{V'},x^\omega_{I'})\,d\mu^{V+I'}_{E}\\
\end{align*}
In particular, there must be some $n$ such that
\[0<\int\prod_{\omega\in\{0,1\}^{I'}}f_{\omega0}(x_{V'},x^\omega_I,x^0_v)g_{\omega,n,N}(x_{V'},x^\omega_{I'})\,d\mu^{V+I'}_{E}.\]
Consider the functions given by, for each $\omega\in\{0,1\}^{I'}$, setting $f'_{\omega}=f_{\omega0}g_{\omega,n,N}$.  We apply the inductive hypothesis to $I'$, and conclude that $\|\mathbb{E}(f'_{0^{I'}}\mid\mathcal{B}_{V,<V})\|_{L^2(\mu^V_{E})}>0$.  Since $g_{0^{I'},n,N}$ is $\mathcal{B}_{V,V'\cup I'}\subseteq \mathcal{B}_{V,<V}$-measurable, it follows that $\|\mathbb{E}(f_{0^I}\mid\mathcal{B}_{V,<V})\|_{L^2(\mu^V_{E})}>0$ as well.
\end{proof}

We can now give a sparse version of the hypergraph removal lemma:
\begin{reptheorem}{thm:main}
  For every $k$-uniform hypergraph $K$ on vertices $V$ and every constant $\epsilon>0$, there are $\delta,\zeta$ so that whenever $\Gamma$ is a $\zeta,|K|2^{2k}$-ccc $k$-uniform hypergraph and $A\subseteq\Gamma$ with $\frac{hom(K,A)}{|\Gamma^{V}_K|}<\delta$, there is a subset $L$ of $A$ with $|L|\leq\epsilon|\Gamma|$ such that $hom(K,A\setminus L)=0$.
\end{reptheorem}
\begin{proof}
  Suppose not.  Let $K,\epsilon$ be a counterexample.  Since there are no such $\delta,\zeta$, for each $n$ we may choose $k$-uniform hypergraphs $H_n\subseteq\Gamma_n$ with $\Gamma_n$ $1/n,|K|2^{2k}$-ccc and $\frac{hom(K,H)}{|\Gamma^{V}_K|}<1/n$.  Let $\mathfrak{M}$ be the model given by Theorem \ref{ultraproduct}.

Let $V$ be the set of vertices of $K$.  For any disjoint $J_0,J_1\subseteq V$, Theorem \ref{converse_random} implies that $U^{J_0}_\infty(\mu^{J_0}_K)$ and, for $\mu^{J_1}_K$-almost every $x_{J_1}$, $U^{J_0}_\infty(\mu^{J_0}_{K,x_{J_1}})$ are characteristic.  By Theorem \ref{converse_nonprincipal}, it follows that each $U^{J,\mathcal{I}^{\bot}}_\infty(\mu^V_E)$ is characteristic, and so by Lemma \ref{char_implies_regular}, $\mu^V_E$ has $J$-regularity.

For each $I\in K$, let $A_I=\{x_{V}\mid x_I\in A\}$.  Since $\frac{hom(K,A_n)}{|(\Gamma_n)^{V}_K|}\rightarrow 0$, we have $\mu^V_K(\bigcap_{I\in K}A_I)=0$.  Then by Theorem \ref{counting}, there must be definable sets $B_I$ such that $\mu^I_K(A_I\setminus B_I)<\epsilon/|K|$ and $\bigcap_{I\in K}B_I=\emptyset$.  Let $L=\bigcup_{I\in K}(A_I\setminus B_I)$, so $\mu^{[1,k]}_{\{[1,k]\}}(L)<\epsilon$.  $L$ is definable from parameters in $M$, and therefore the properties $\mu^{[1,k]}_{\{1,k\}}(L)<\epsilon$ and $\bigcap_{I\in K}(A_I\setminus L)=\emptyset$ are witnessed by formulas.  Therefore there must be arbitrarily large finite models where these formulas are satisfied.  But this contradicts the choice of the hypergraphs $H_n,\Gamma_n$.
\end{proof}

\section{Conclusion}

Many other notions of pseudorandomness for hypergraphs that have been considered \cite{chung:MR1138430,chung:MR1068494,kohayakawa:MR1883869,conlon09}.  The next step towards developing a rich analytic approach to working with sparse random hypergraphs would be a detailed investigation of the relationship between notions of pseudorandomness in the finite setting and the corresponding properties of measures in the infinitary setting.  With weaker notions of pseudorandomness, we would expect to lose the full Fubini property, but the notions that replace it are likely to be of interest themselves.

Another interesting direction would be to weaken the notion of randomness to allow the ambient hypergraph to be, say, bipartite.  This is needed if one hopes to use these methods to prove sparse versions of Szemer\'edi's Theorem in the style of the Green-Tao theorem on arithmetic progressions in the primes \cite{green:MR2415379}.  (Conlon, Fox, and Zhou have recently extended their version of sparse graph removal to precisely such a proof \cite{MR3361771}.)

The approach Conlon and Gowers use to prove hypergraph regularity \cite{conlon:2011} depends, like our approach, on the use of various norms to detect the presence of certain properties.  Their norms are much more narrowly tailored than the general uniformity norms.  The uniformity norms are strikingly natural in the infinitary setting, lining up with canonical algebras of definable sets; it is possible that other norms also correspond to algebras which might be of independent interest.

\printbibliography

\begin{dajauthors}
\begin{authorinfo}[hpt]
  Henry Towsner\\
  Department of Mathematics\\
  University of Pennsylvania\\
  209 South 33rd Street, Philadelphia, PA 19104-6395, USA\\
  htowsner\imageat{}gmail\imagedot{}com \\
  \url{http://www.math.upenn.edu/~htowsner}
\end{authorinfo}
\end{dajauthors}

\end{document}